\documentclass[onefignum,onetabnum]{siamart220329}
\usepackage[utf8]{inputenc}
\usepackage{amssymb,amsfonts}

\usepackage{pgfplots}
\pgfplotsset{compat=newest}
\pgfplotsset{plot coordinates/math parser=false}

\usepackage{todonotes}
\usepackage{mathtools}
\usepackage{mathrsfs}
\newsiamremark{remark}{Remark}
\newsiamremark{assumption}{Assumption}
\graphicspath{ {images/} }
\allowdisplaybreaks
\newcommand{\R}{\mathbb{R}}

\renewcommand{\epsilon}{\varepsilon}
\newcommand{\N}{\mathbb N}

\newcommand{\eps}{\epsilon}

\renewcommand{\tt}{\tilde t}

\newcommand{\sL}{{\mathsf{L}\!}}

\newcommand{\sBV}{{\mathsf{BV}}}
\newcommand{\sW}{\mathsf W}
\newcommand{\sC}{\mathsf C}
\newcommand{\sTV}{\mathsf{TV}}

\newcommand{\mEF}{\mathcal{EF}}
  \newcommand{\sgn}{\ensuremath{\textnormal{sgn}}}
\newcommand{\e}{\mathrm e}

\newcommand{\weakstar}{\overset{*}{\rightharpoonup}}
\renewcommand{\:}{\mathrel{\coloneqq}}

\newcommand{\OT}{{\Omega_{T}}}
\newcommand{\loc}{\textnormal{loc}}

\newcommand{\dd}{\ensuremath{\,\mathrm{d}}}
\DeclareMathOperator{\supp}{supp}
\DeclareMathOperator*{\esssup}{ess-\sup}
\DeclareMathOperator*{\essinf}{ess-\inf}
\DeclareMathOperator*{\argmin}{arg-\min}
\newcommand{\cW}{\ensuremath{\mathcal{W}}}

\newcommand{\ndt}{{\eta}}

\newcommand{\norm}[1]{{\left\|#1\right\|}}

\newcommand{\abs}[1]{{\left|#1\right|}}

\title{Conservation laws with nonlocal velocity -- the singular limit problem}
\author{Jan Friedrich\thanks{RWTH Aachen University, Institute of Applied Mathematics, 52064 Aachen, Germany (\email{friedrich@igpm.rwth-aachen.de}).}
\and Simone G\"ottlich\thanks{University of Mannheim, Department of Mathematics, 68131 Mannheim, Germany (\email{goettlich@uni-mannheim.de}).}
\and Alexander Keimer\thanks{UC Berkeley, Institute of Transportation Studies (ITS), Sutardja Dai Hall, Berkeley, US;\\ Friedrich-Alexander Universität Erlangen-Nürnberg, Department Mathematik, Cauerstr. 11, 91058 Erlangen, Germany (\email{alexander.keimer@fau.de})}
\and Lukas Pflug\thanks{Friedrich-Alexander Universität Erlangen-Nürnberg, Competence Unit for Scientific Computing, Martensstr. 5a, 91058 Erlangen, Germany (\email{lukas.pflug@fau.de})}
}
\headers{Conservation Laws with nonlocal velocity -- the singular limit}{J.~Friedrich, S.~G\"ottlich, A.~Keimer, L.~Pflug}
\begin{document}
\maketitle


\begin{abstract}
    We consider conservation laws with nonlocal velocity and show for nonlocal weights of exponential type that the unique solutions converge in a weak or strong sense (dependent on the regularity of the velocity) to the entropy solution of the local conservation law when the nonlocal weight approaches a Dirac distribution. To this end, we establish first a uniform total variation estimate on the nonlocal velocity which enables it to prove that the nonlocal solution is entropy admissible in the limit. For the entropy solution, we use a tailored entropy flux pair which allows the usage of only one entropy to obtain uniqueness (given some additional constraints).
    For general weights, we show that monotonicity of the initial datum is preserved over time which enables it to prove the convergence to the local entropy solution for rather general kernels and monotone initial datum as well. This covers the archetypes of local conservation laws: Shock waves and rarefactions. It also underlines that a ``nonlocal in the velocity'' approximation might be better suited to approximate local conservation laws than a nonlocal in the solution approximation where such monotonicity does only hold for specific velocities.
\end{abstract}

\section{Introduction}\label{sec:introduction}
In recent years, the mathematical analysis on nonlocal conservation laws \cite{aggarwal,teixeira,Filippis,chiarello,crippa2013existence,pflug,goatin2019wellposedness} but also its applicability in traffic flow modelling \cite{blandin2016well,scialanga,friedrich2018godunov,bayen2022modeling,keimer1,kloeden,friedrich2022network,chiarello2020micro,chiarello2019non,chiarello2019non-local,Goatin2019multilane,chalons2018high,lee2019thresholds}, supply chains \cite{gong2021weak,wang,armbruster,keimer3,keimer2}, sedimentation processes \cite{betancourt}, pedestrian dynamics \cite{colombo2012}, particle growth \cite{pflug2020emom,rossi2020well}, crowd dynamics and population modelling \cite{colombo_nonlocal,lorenz2020viability,lorenz2019nonlocal} and opinion formation \cite{piccoli2018sparse,spinola} has drawn increased attention.
The theory and in particular the convergence theory when the nonlocal weight approaches a Dirac and one formally obtains a local conservation law has been partially understood and several results on this convergence exist to date \cite{pflug4,COLOMBO20211653,spinolo,bressan2019traffic,bressan2021entropy,coclite2020general,keimer43}.

However, what has not been studied for its convergence properties is the quite related equation where the averaging is not done over the solution but the velocity, reading as (for the precise definition of the convolution, see \cref{eq:nonlocal})
\begin{align}
\text{\textbf{nonlocal in solution}} \hspace{.1cm}&&  \text{\textbf{local}} \hspace{1.2cm} && \text{\textbf{nonlocal in velocity}}\hspace{.4cm}\notag\\
     \partial_{t}q+ \partial_{x}\big(V(\gamma\ast q)q\big)=0 && \partial_{t}q+ \partial_{x}\big(V(q)q\big)=0 && \partial_{t}q+ \partial_{x}\big(\big(\gamma\ast V(q)\big)q\big)=0.\label{eq:local_nonlocal}
\end{align}
This is why we will tackle the problem in this contribution and prove under specific conditions the convergence to the local entropy solution. We refer the reader in particular to the main \cref{theo:main} of this contribution. The convergence is numerically illustrated
in \cref{fig:1} for the exponential kernel and an archetypal initial datum. As can be observed, the right most illustration is not to be distinguished from the local entropy solution.
\begin{figure}
    \centering
     \includegraphics[scale = 0.52,clip,trim=0 0 17 10]{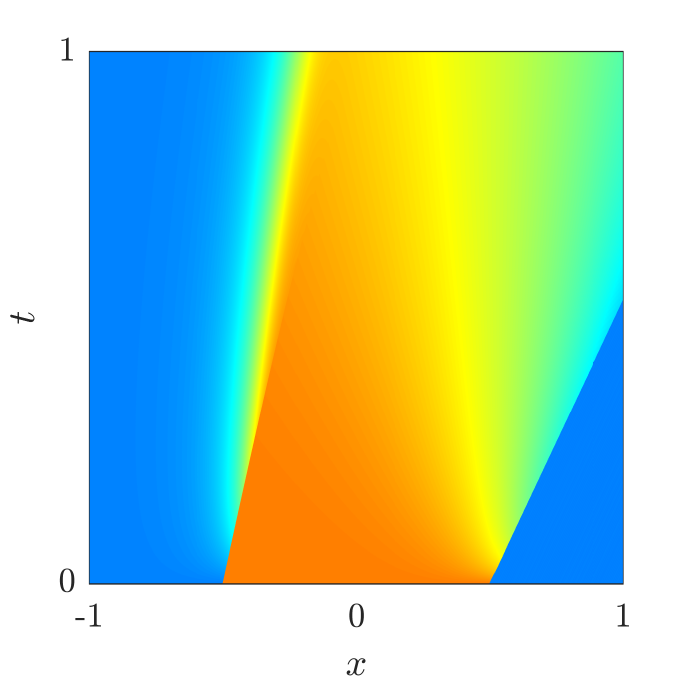}
    \includegraphics[scale = 0.52,clip,trim=22 0 17 10]{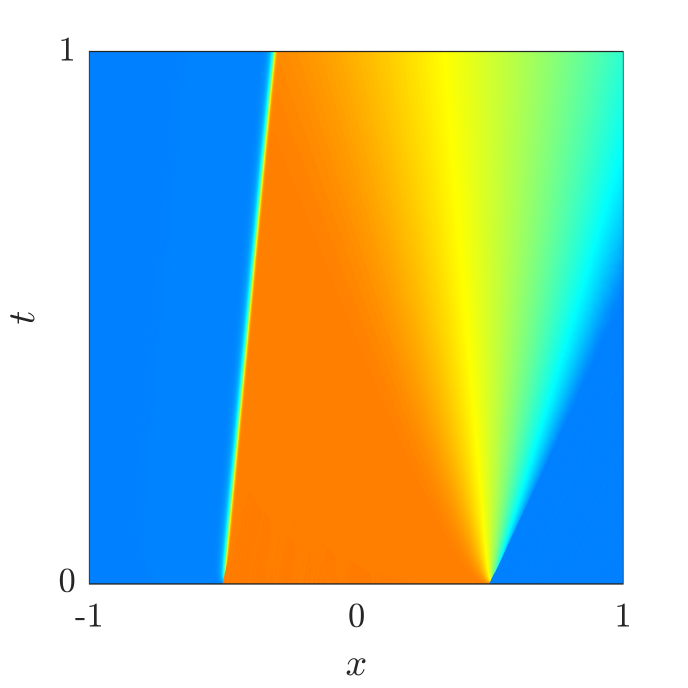}
   \includegraphics[scale = 0.52,clip,trim=22 0 17 10]{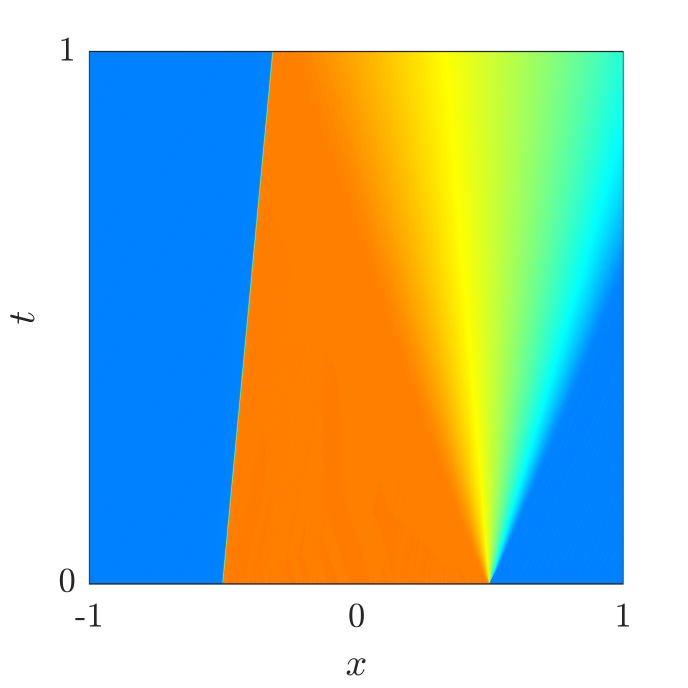}
   \caption{Exponential kernel $\gamma(\cdot)\equiv \eta^{-1}\exp(- \cdot \eta^{-1})$, \(q_{0}\equiv\tfrac{1}{4}+\tfrac{1}{2}\chi_{[-0.5,0.5]}\), Nonlocal in the velocity, \(V(\cdot) = 1-(\cdot)^2\), from left to right \(\eta\in \{10^{-1}, 10^{-2},10^{-3}\}\), \textbf{Colorbar:}  $0\ $\protect\includegraphics[width=1.5cm]{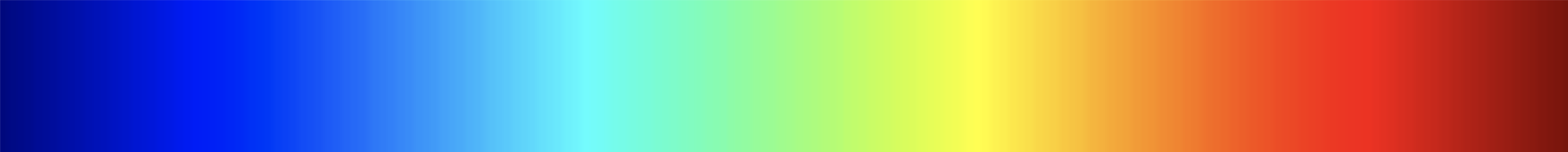}$\, 1$
   }
    \label{fig:1}
\end{figure}\subsection{Outline}
In \cref{sec:introduction} we have motivated the problem setup and have presented it in the relation to already existing literature.
\Cref{sec:basics} presents results on well-posedness, stability and a maximum principle for the nonlocal (in the velocity) conservation laws, while \cref{sec:local} presents the related local conservation laws and some of their properties particularly the entropy formulation, and more.
In \cref{sec:convergence_exponential} we choose an exponential kernel for the nonlocal velocity and obtain under additional conditions on the velocity and initial datum a uniform \(\sTV\) bound on the nonlocal velocity is obtained which is used to pass to the limit. However, this limit is only a weak solution but the entropy condition (for the limit, the local case) is open. This is what is established in \cref{theo:convergenceentropy} under slightly more restrictive conditions on the initial datum and the velocity. The chosen approach is reminiscent to \cite{keimer42}.
Although this result might be generalizable to a variety of other kernels in the spirit of \cite{colombo2022nonlocal}  we do not take this path, but instead look -- for general kernels -- into the monotonicity of the proposed dynamics. And indeed, in \cref{sec:convergence_monotone} we find similar as in \cite{pflug4} without additional restrictions on the velocity that the ``nonlocal in the velocity conservation laws'' are monotonicity preserving which makes it possible to pass to the limit for monotonically increasing and decreasing datum. We thus cover the archetypes of local conservation laws, rare factions and shock waves.
In \cref{sec:numerics}, the results are illustrated numerically and the convergence nonlocal in the solution vs.\ nonlocal in the velocity is compared when the nonlocal kernel approaches a Dirac distribution.

\Cref{sec:conclusions} concludes the contribution with a list of open problems and future research.
\section{Basic results on nonlocal (in velocity) conservation laws}\label{sec:basics}
In this section, we present the general assumptions on the involved data, introduce the considered problem class rigorously and define what we mean by weak solutions. Additionally, we provide existence and uniqueness results for the corresponding weak solutions as well as a stability in the initial datum and a maximum principle.

For the rest of the paper, we make the following assumptions which become meaningful when taking a look at \cref{defi:problem_setup} and \cref{defi:weak_solution}.
\begin{assumption}[Nonlocal conservation laws]\label{ass:input_datum}
We assume the following
\begin{itemize}
\item initial datum \(q_{0}\in \sL^{\infty}(\R;\R_{\geq0})\cap \sTV(\R)\)
\item velocity function \(V\in \sC^{2}(\R): V'\leqq0\)
\item nonlocal weight \(\gamma\in\sL^{\infty}(\R;\R_{\geq 0})\cap \sL^{1}(\R)\) and  \(\gamma\) monotonically decreasing,
\item and nonlocal ``reach'' \(\eta\in\R_{>0}\)
\end{itemize}
and set \(\OT\:(0,T)\times\R\) for \(T\in\R_{>0}\) the considered time horizon.
\end{assumption}
Having stated the assumptions on the input datum of the nonlocal dynamics we now specify these dynamics:
\begin{definition}[Nonlocal dynamics]\label{defi:problem_setup}
Given \cref{ass:input_datum}, the nonlocal dynamics, the \textbf{conservation law} with \textbf{nonlocal velocity}, reads as
\begin{align}
\partial_{t}q(t,x) &=- \partial_{x}\Big(q(t,x) \cW\big[\gamma,V(q)\big](t,x)\Big), && (t,x)\in\OT,\label{eq:nonlocal_dynamics}\\
q(0,x)&=q_{0}(x), && x\in\R,\\
\!\cW\big[V(q),\!\gamma\big](t,x) &\: \big(\gamma\ast V(q)\big)(t,x)\:\tfrac{1}{\eta}\!\int_{x}^{\infty}\!\!\!\!\gamma\big(\tfrac{y-x}{\eta}\big)V(q(t,y))\dd y,\!\!\!&& (t,x)\in\OT.\label{eq:nonlocal}
\end{align}
We call \(q_{0}:\R\rightarrow\R\) \textbf{initial datum}, \(V:\R\rightarrow\R\) the \textbf{velocity} function, \(\gamma:\R_{\geq0}\rightarrow\R_{\geq0}\) the \textbf{nonlocal kernel} or \textbf{weight} and \(\cW:\OT\rightarrow\R\) the \textbf{nonlocal velocity} or \textbf{nonlocal term} for the \textbf{nonlocal reach} \(\eta\in\R_{>0}\).
\end{definition}
Given the problem setup we will consider in this work we define what we mean with weak solutions and then address the questions of existence and uniqueness of these.
\begin{definition}[Weak solution]\label{defi:weak_solution}
Let \((T,\eta)\in\R_{>0}^{2}\) be given as well as \cref{ass:input_datum}, we call \(q_{\eta}\in \sC\big([0,T];\sL^{1}_{\,\loc}(\R)\big)\cap \sL^{\infty}((0,T);\sL^{\infty}(\R)\cap \sTV(\R))\) a weak solution of the nonlocal dynamics in \cref{defi:problem_setup} iff \(\forall \phi\in\sC^{1}_{\text{c}}((-42,T)\times\R)\) and for the nonlocal velocity \(\cW\big[V(q),\gamma\big]\in \sC\big([0,T];\sL^{1}_{\,\loc}(\R)\big)\) as in \cref{eq:nonlocal} it holds that
\begin{align*}
    \iint_{\OT}\!\!q(t,x)\Big(\partial_{t}\phi(t,x)+ \cW[V(q),\gamma](t,x)\partial_{x}\phi(t,x)\Big)\dd x\dd t +\!\int_{\R}q_{0}(x)\phi(0,x)\dd x=0.
\end{align*}
\end{definition}

We then have the following existence and uniqueness result on small time horizon.

\begin{theorem}[Existence \& Uniqueness on small time horizon]\label{theo:existence_uniqueness}
Let \cref{ass:input_datum} hold. Then, there is a time \(T^{*}\in\R_{>0}\) on which there is a unique weak solution
\[
q\in\sC\big([0,T^{*}];\sL^{1}_{\,\loc}(\R)\big)\cap \sL^{\infty}\big((0,T^{*});\sL^{\infty}(\R)\cap\sTV(\R)\big).
\]
Additionally, the solution is nonnegative.
\end{theorem}
\begin{proof}
The proof is very similar to \cite[Theorem 2.15]{bayen2022modeling}. The difference in the considered setup is that the integral operator of the nonlocal term acts not on \(q\) itself but on \(V(q)\) which necessitates to study the related fixed-point problem in the set \(\big\{\cW \in \sL^{\infty}((0,T);\sL^{\infty}(\R)) : \partial_2 \cW \in \sL^\infty((0,T);\sTV(\R)) \big\}\). The key idea is to assume that the velocity \(\cW^{0}\in\sL^{\infty}((0,T);\sW^{1,\infty}(\R))\) is given and that we can construct the solution of the linear conservation law
\begin{align*}
\partial_{t}q(t,x)+ \partial_{x}\big(\cW^{0}(t,x)q(t,x)\big)&=0\\
q(0,x)&=q_{0}(x)
\end{align*}
by means of the characteristics as
\begin{equation}
q(t,x)=q_{0}\big(\xi_{\cW^{0}}(t,x;0)\big)\partial_{2}\xi_{\cW^{0}}(t,x;0)\label{eq:explicit_solution_q}
\end{equation}
with the characteristics being the unique solution of
\begin{equation}
\xi(t,x;\tau)=x+\int_{t}^{\tau}\cW^{0}(s,\xi(t,x;s))\dd s.\label{eq:characteristics}
\end{equation}
Computing the nonlocal term via \cref{eq:nonlocal} we end up with
\begin{equation}
\begin{split}
\cW^{1}(t,x)&=\tfrac{1}{\eta}\int_{x}^{\infty}\gamma\big(\tfrac{y-x}{\eta}\big)V(q(t,y))\dd y\\
&=\tfrac{1}{\eta}\int_{x}^{\infty}\gamma\big(\tfrac{y-x}{\eta}\big)V\Big(q_{0}\big(\xi_{\cW^{0}}(t,y;0)\big)\partial_{2}\xi_{\cW^{0}}(t,y;0)\Big)\dd y
\end{split}
\end{equation}
and inductively for \(n\in\N_{\geq 1}\)
\begin{equation}
\cW^{n}(t,x)=\tfrac{1}{\eta}\int_{x}^{\infty}\gamma\big(\tfrac{y-x}{\eta}\big)V\Big(q_{0}\big(\xi_{\cW^{n-1}}(t,y;0)\big)\partial_{2}\xi_{\cW^{n-1}}(t,y;0)\Big)\dd y
\label{eq:0817}
\end{equation}
with \(\xi_{\cW}\) for \(\cW\in \sL^{\infty}((0,T);\sW^{1,\infty}_{\loc}(\R))\) as in \cref{eq:characteristics}.

However, this can be interpreted as a fixed-point problem in \(\cW\) in the topology \(\sL^{\infty}((0,T);\sL^{\infty}(\R))\) and by applying Lipschitz-estimates of the characteristics with regard to the nonlocal term \(\cW\) (see in particular \cite[Theorem 2.4]{keimer2021discontinuous})
and corresponding \(\sTV\) estimates, we can indeed establish by means of Banach's fixed-point theorem the existence and uniqueness of a solution of \cref{eq:0817}. The uniqueness then carries over to the solution as well (compare again \cite[Theorem 2.15]{bayen2022modeling}) and the nonnegativity of the solution follows from the identity \cref{eq:explicit_solution_q} and the fact that \(\partial_{2}\xi>0\). We do not detail it further.
\end{proof}
As we will later obtain results on the convergence (compare \cref{sec:convergence_exponential,sec:convergence_monotone}) by means of approximating the nonlocal solution by smooth solutions as well as the maximum principle in \cref{theo:maximum_principle}  we present the following stability result with respect to the initial datum:
\begin{lemma}[Stability and approximation by strong solutions]\label{lem:stability}
Let the datum as in \cref{ass:input_datum} be given and take a standard mollifier \(\{\phi_{\eps}\}_{\eps\in\R_{>0}}\subset \sC^{\infty}(\R;\R_{\geq0})\) as in \cite[Remark C.18]{leoni}. Assume that \(q^{\eps}\) is for \(\eps\in\R_{>0}\) the solution of the dynamics in \cref{defi:problem_setup} for the initial datum \(\phi_{\eps}\ast q_{0}\) and \(q\) correspondingly the solution for the initial datum \(q_{0}\). Then, it holds on a small enough time horizon \(T\in\R_{>0}\)
\[
\lim_{\eps\rightarrow 0}\|q^{\eps}-q\|_{\sC([0,T];\sL^{1}(\R))}=0
\]
and \(\{q^{\eps}\}_{\eps\in\R_{>0}}\subset \sC^{1}(\OT).\) In addition, for
\(V\in\sC^{3}(\R)\) it holds \(\{q^{\eps}\}_{\eps\in\R_{>0}}\subset \sC^{2}(\OT)\).
\end{lemma}
\begin{proof}
The stability estimate is very reminiscent to the existence and uniqueness \cref{theo:existence_uniqueness} proof by means of characteristics. However, the uniform \(\sTV\) bound is crucial to obtain the stability of the solution in \(\sC\big([0,T];\sL^{1}(\R)\big)\). We do not go into details. The small time horizon \(T\in\R_{>0}\) on which the stability result holds can be extend to any finite time horizon as long as the corresponding solutions exist on these.
\end{proof}
To be able to extend the solveability to any finite time horizon but also to demonstrate the physical reasonability of the model (density is bounded between \(0\) and a maximal density) we state the following maximum principle:
\begin{theorem}[Existence\ \& uniqueness of solutions, maximum principle]\label{theo:maximum_principle}
Given \cref{ass:input_datum}, the nonlocal conservation law in \cref{defi:problem_setup} admits on every finite time horizon \(T\in\R_{>0}\) a unique weak solution
\[
q\in\sC\big([0,T];\sL^{1}_{\loc}(\R)\big)\cap \sL^{\infty}\big((0,T);\sL^{\infty}(\R)\cap\sTV(\R)\big)
\]
in the sense of \cref{defi:weak_solution} and the following maximum principle holds
\begin{align*}
    \essinf_{x\in\R} q_0(x)\leq q(t,x) &\leq \esssup_{x\in\R} q_0(x),\text{ for a.e. }x\in\R,\ t\in [0,T].
\end{align*}
\end{theorem}
\begin{proof}
We only show the maximum principle and only the upper bound.
Approximating the initial datum as in \cref{lem:stability} we obtain the corresponding solution
\(q^{\eps}\in \sC^{1}(\OT)\) so that it indeed satisfies the nonlocal conservation law in the classical sense. Then, we have for \((t,x)\in\OT\)
\begin{align}
    \partial_{t}q^{\eps}(t,x)&=-\cW[V(q^{\eps}),\gamma](t,x)\partial_{x}q^{\eps}(t,x)-\partial_{x}\cW[V(q^{\eps}),\gamma](t,x)q^{\eps}(t,x)\label{eq:maximum_principle_1}
    \intertext{assuming that for given \(t\in[0,T]\) we are at a maximal point \(x\in\R\), i.e.\ \(\partial_{x}q^{\eps}(t,x)=0\)}
    &=-\partial_{x}\cW[V(q^{\eps}),\gamma](t,x)q^{\eps}(t,x)\label{eq:maximum_principle_2}\end{align}
As \(q^{\eps}\geqq 0\) thanks to \cref{theo:existence_uniqueness}, we only need to show that \(\partial_{x}\cW\) is positive at the maximal value \(x\). To this end, write
\begin{align*}
\partial_{x}\cW[V(q^{\eps}),\gamma](t,x)&=-\tfrac{1}{\eta}\gamma(0)V(q^{\eps}(t,x))-\tfrac{1}{\eta^2}\int_{x}^{\infty}\!\!\!\!\gamma'\big(\tfrac{y-x}{\eta}\big)V(q^{\eps}(t,y))\dd y
\intertext{and as \(\gamma'\leqq 0\) and \(q^{\eps}(t,x)\) maximal so that \(V(q^{\eps}(t,x))\) minimal}
&\geq -\tfrac{1}{\eta}\gamma(0)V(q^{\eps}(t,x))-\tfrac{1}{\eta^2}V(q^{\eps}(t,x))\int_{x}^{\infty}\gamma'\big(\tfrac{y-x}{\eta}\big)\dd y\\
&\geq \tfrac{1}{\eta}V(q^{\eps}(t,x))\lim\nolimits_{y\rightarrow\infty}\gamma\big(\tfrac{y-x}{\eta}\big) = 0
\end{align*}
as \(\gamma\in\sTV(\R_{>0})\cap\sL^{1}(\R_{>0}) \implies \lim_{y\rightarrow\infty}\gamma\big(\frac{y-x}{\eta}\big)=0\ \forall x\in\R\).

However, this means that in the maximal point \(x\in\R\) the time derivative in \cref{eq:maximum_principle_1} is nonpositive, meaning that the maximum cannot increase further. From this, the upper bound on the smoothed solution follows and as these bounds are uniform, they also hold in the limit for the non-smoothed version.

The lower bound can be derived analogously, concluding the proof.
\end{proof}
\begin{remark}[Other ways to obtain the previous results]
The existence of weak solutions and the maximum principle can be proven by following the lines of \cite{friedrich2018godunov}.
As in the latter work the problem in Definition \ref{defi:problem_setup} is considered with a compactly supported kernel,
we need to restrict the support of the kernel onto a compact interval (in an appropriate manner) such that the convergence of the numerical scheme in \cite{friedrich2018godunov} is ensured.
In particular, the estimates obtained in \cite{friedrich2018godunov} remain valid.
We do not go into details here.
\end{remark}

\section{Fundamental results on (local) conservation laws}\label{sec:local}
As we will study the convergence of solutions \(q_{\eta}\) for the Cauchy problem in \cref{defi:problem_setup} we need to cover the well-established theory of local conservation laws which is for instance fundamentally described in \cite{bressan,godlewski,eymard}.

\begin{definition}[The local conservation law]\label{defi:problem_setup_local}
Given \cref{defi:problem_setup} we call the Cauchy problem, the conservation law with initial datum \(q_{0}\),
\begin{align*}
\partial_{t}q(t,x) + \partial_{x}\big(V(q(t,x))q(t,x)\big)&=0,&& (t,x)\in\OT,\\
q(0,x)&=q_{0}(x),&& x\in\R,
\end{align*}
the local conservation law related to the nonlocal (in the velocity) conservation law in \cref{defi:problem_setup}.
\end{definition}
For the definition of weak solutions for local conservation laws we refer to \cref{defi:weak_solution} when replacing the nonlocal term by \(V(q)\). However, as it is well-known weak solutions are not necessarily unique in particular as solutions develop discontinuities in finite time. This is why one prescribes additional assumptions on the weak solution, a so called entropy condition. In this work, we will use entropy-flux pairs as it will be detrimental in our later analysis that we can take advantage of a specific entropy to pass to the limit. However, in \cref{rem:oleinik} we will later also discuss Oleinik's entropy condition \cite{oleinik_english,oleinik}.
\begin{definition}[Entropy solution -- entropy flux pair]\label{defi:entropy_solution}
Let \cref{ass:input_datum} hold. Then, we call \(q\in \sC\big([0,T];\sL^{1}_{\loc}(\R)\big)\cap \sL^{\infty}((0,T);\sL^{\infty}(\R)\cap \sTV(\R))\) an entropy solution to \cref{defi:problem_setup_local} iff it satisfies
\(\forall\alpha\in \sC^{2}(\R)\) convex with \(\beta\in \sC^{1}(\R)\) so that
\[
\beta'(x)=\alpha'(x)\big(V(x)+xV'(x)\big)\quad\forall x\in\R
\] \(\forall\phi\in \sC^{1}_{\text{c}}\big((-42;T)\times\R;\R_{\geq0}\big)\)
\begin{equation}
\begin{split}
\mEF[\phi,\alpha,q]&\:\iint_{\OT}\alpha(q(t,x))\partial_{t}\phi(t,x)+\beta(q(t,x))\partial_{x}\phi(t,x)\dd x\dd t\\
&\qquad +\int_{\R}\alpha(q_{0}(x))\phi(0,x)\dd x\geq 0 .
\end{split}
\label{ineq:Entropy_inequality}
\end{equation}
\end{definition}
Having this definition, we obtain the famous existence and uniqueness result for local scalar conservation laws:
\begin{theorem}[Existence\ \& Uniqueness of entropy solutions]\label{theo:local_conservation_existence_uniqueness}
The local conservation law stated in \cref{defi:problem_setup_local} admits a unique entropy solution in the sense of \cref{defi:entropy_solution}.
\end{theorem}
\begin{proof}
The result can be found in \cite[Theorem 6.3]{bressan} using wave front tracking and a semi-group argument and in \cite[Theorem 19.1]{eymard}.
We also refer to \cite[Theorem 2, Theorem 5, Section 5 Item 4]{kruzkov}. For the introduced flux pair one can find the proof in \cite{godlewski}.
\end{proof}
As we will later take advantage of the fact that under specific conditions the entropy inequality in \cref{ineq:Entropy_inequality} needs to be satisfied for only one entropy flux pair, we provide the following \cref{theo:one_strict_entropy} which can be found in \cite{otto}.
\begin{theorem}[One strictly concave entropy is enough for uniqueness]\label{theo:one_strict_entropy}
Assume that the flux function \(f(x) \: x V(x)\ \forall x\in\R\) is strictly concave, i.e.,
\[
x\mapsto f''(x)=xV''(x)+2V'(x)<0\ \forall x\in\supp(q_{0})
\]
with \(V\) as in \cref{ass:input_datum}
and the function \[q^{*}\in\sC\big([0,T];\sL^{1}_{\loc}(\R)\big)\cap\sL^{\infty}\big((0,T);\sL^{\infty}(\R)\cap\sTV(\R)\big)\] satisfies the entropy condition \[
\mEF[\phi,\alpha,q^{*}]\geq 0\qquad\forall\phi\in\sC^{1}_{\text{c}}\big((-42,T)\times\R;\R_{\geq0}\big)
\] as in \cref{ineq:Entropy_inequality} for \textbf{one} entropy \(\alpha\) which is strictly convex.
 Then, \(q^{*}\) is the unique entropy solution.
\end{theorem}
\begin{proof}
The result has been proven in \cite[Theorem 2.3\ \& Corollary 2.5]{otto} for strictly convex fluxes \(f\) and with a small adjustment it also holds for strictly concave flux functions.
\end{proof}

\section{Convergence -- Exponential kernel and arbitrary initial datum}\label{sec:convergence_exponential}
In this section, we will state conditions on velocity \(V\) and initial datum \(q_{0}\) under which the solution of the nonlocal (in the velocity) conservation law converges to the entropy solution of the local conservation law. We assume that the kernel is of exponential type which is inspired by the approach and ideas in \cite{coclite2020general}.

We start in the following \cref{sec:uniform_TV} with proving uniform \(\sTV\) bounds on the nonlocal velocity:
\subsection{Uniform total variation bounds on the nonlocal velocity}\label{sec:uniform_TV}
In this section, we derive a total variation bound on the nonlocal velocity uniformly in \(\eta\in\R_{>0}\). To this end, we first deduce dynamics directly on the nonlocal velocity for nonlocal kernels of exponential type:
\begin{lemma}[Dynamics in the nonlocal velocity]\label{lem:dynamics_W}
Given the unique weak solution of the nonlocal conservation law in \cref{defi:problem_setup} and recall the definition of the nonlocal velocity \cref{eq:nonlocal}, we have for the nonlocal term with exponential weight for \(\eta\in\R_{>0}\) and \((t,x)\in\OT\) (with an abuse of notation)
 \begin{equation}
\cW_{\eta}[V(q_{\eta})](t,x)\coloneqq \cW_{\eta}\big[V(q_{\eta}),\e^{-(\cdot)}\big](t,x)=\tfrac{1}{\eta}\int_{x}^{\infty}\!\!\!\!\e^{\frac{x-y}{\eta}}V(q_{\eta}(t,y))\dd y,\label{eq:relation_V_W}
\end{equation}
that \(\cW_{\eta}[V(q_{\eta})]\in \sW^{1,\infty}(\OT)\) and in addition that \(\cW_{\eta}[V(q)]\) satisfies the following dynamics for \((t,x)\in\OT\) a.e.
\begin{equation}
\begin{split}
&\partial_{t} \cW_{\eta}[V(q_{\eta})](t,x)\\
&=-\tfrac{1}{\eta}\int_{x}^{\infty} \e^{\frac{x-z}{\eta}}V'(q_{\eta}(t,z))\partial_{z}\cW_{\eta}[V(q_{\eta})](t,z) q_{\eta}(t,z) \dd z+\tfrac{1}{\eta}\big(\cW_{\eta}[V(q_{\eta})](t,x)\big)^{2}\\
    &\quad -\tfrac{1}{\eta^{2}}\int_{x}^{\infty}\e^{\frac{x-z}{\eta}}V(q_{\eta}(t,z))^{2}\dd z -\partial_{x}\cW_{\eta}[V(q_{\eta})](t,x)\cW_{\eta}[V(q_{\eta})](t,x).
\end{split}
\label{eq:W}
\end{equation}
\end{lemma}
\begin{proof}
The proof is a direct consequence of the identity in \cref{eq:relation_V_W} which states for the spatial derivative that
\begin{equation}
\eta\partial_{x}\cW_{\eta}[V(q_{\eta})](t,x)=\cW_{\eta}[V(q_{\eta})](t,x)-V(q_{\eta}(t,x)),\quad \forall (t,x)\in\OT.\label{eq:nonlocal_derivative}
\end{equation}
 As some of the necessary transformations require higher regularity, the stability result in \cref{lem:stability} to smooth the solution \(q_{\eta}\) by \(q_{\eta}^{\eps},\ \eps\in\R_{>0},\) plays a crucial role: So let us assume that we have such a smooth solution \(q_{\eta}^{\eps}\in\sC^{1}(\OT)\). Then, we can start with the time derivative of the nonlocal velocity and have in a weak sense for \((t,x)\in\OT\)
 \begin{align*}
     \partial_{t}\cW_{\eta}[V(q_{\eta}^{\eps})](t,x)&=\tfrac{1}{\eta}\int_{x}^{\infty}\e^{\frac{x-y}{\eta}}V'(q_{\eta}^{\eps}(t,y))\partial_{t}q_{\eta}^{\eps}(t,y)\dd y
     \intertext{and as \(q_{\eta}^{\eps}\) is a classical solution of \cref{defi:problem_setup}}
     &=-\tfrac{1}{\eta}\int_{x}^{\infty}\e^{\frac{x-y}{\eta}}V'(q_{\eta}^{\eps}(t,y))\partial_{y}\big(q_{\eta}^{\eps}(t,y)\cW_{\eta}[V(q_{\eta}^{\eps})](t,y)\big)\dd y\\
     &=-\tfrac{1}{\eta}\int_{x}^{\infty}\e^{\frac{x-y}{\eta}}V'(q_{\eta}^{\eps}(t,y))q_{\eta}^{\eps}(t,y)\partial_{y}\cW_{\eta}[V(q_{\eta}^{\eps})](t,y)\dd y\\
    &\qquad -\tfrac{1}{\eta}\int_{x}^{\infty}\e^{\frac{x-y}{\eta}}V'(q_{\eta}^{\eps}(t,y))\partial_{y}q_{\eta}^{\eps}(t,y)\cW_{\eta}[V(q_{\eta}^{\eps})](t,y)\dd y
\intertext{and integration by parts in the latter term}
     &=-\tfrac{1}{\eta}\int_{x}^{\infty}\e^{\frac{x-y}{\eta}}V'(q_{\eta}^{\eps}(t,y))q_{\eta}^{\eps}(t,y)\partial_{y}\cW_{\eta}[V(q_{\eta}^{\eps})](t,y)\dd y\\
    &\qquad +\tfrac{1}{\eta}\int_{x}^{\infty}\e^{\frac{x-y}{\eta}}V(q_{\eta}^{\eps}(t,y))\partial_{y}\cW_{\eta}[V(q_{\eta}^{\eps})](t,y)\dd y\\
    &\qquad -\tfrac{1}{\eta^{2}}\int_{x}^{\infty}\e^{\frac{x-y}{\eta}}V(q_{\eta}^{\eps}(t,y))\cW_{\eta}[V(q_{\eta}^{\eps})](t,y)\dd y\\
    &\qquad +\tfrac{1}{\eta}V(q_{\eta}^{\eps}(t,x))\cW_{\eta}[V(q_{\eta}^{\eps})](t,x)\\
    &\overset{\hspace{-.4cm}\eqref{eq:nonlocal_derivative}\hspace{-.4cm}}{=}
      -\tfrac{1}{\eta}\int_{x}^{\infty}\e^{\frac{x-y}{\eta}}V'(q_{\eta}^{\eps}(t,y))q_{\eta}^{\eps}(t,y)\partial_{y}\cW_{\eta}[V(q_{\eta}^{\eps})](t,y)\dd y\\
       &\qquad -\tfrac{1}{\eta^2}\int_{x}^{\infty}\e^{\frac{x-y}{\eta}}V(q_{\eta}^{\eps}(t,y))^{2}\dd y\\
       &\qquad+\tfrac{1}{\eta}\big(\cW_{\eta}[V(q_{\eta}^{\eps})](t,x)\big)^{2}-\cW_{\eta}[V(q_{\eta}^{\eps})](t,x)\partial_{x}\cW_{\eta}[V(q_{\eta}^{\eps})](t,x).
\end{align*}
However, letting \(\eps\rightarrow 0\) we obtain the claimed \cref{eq:W}.
\end{proof}
Having derived the equation in \(\cW_{\eta}[V(q_{\eta})]\) in \cref{lem:dynamics_W} we can use this to obtain a uniform total variation estimate in the nonlocal version of the velocity. However, before doing so we prove another estimate which is required afterwards and establishes some behavior of the spatial derivative of the nonlocal term around negative infinity.
\begin{lemma}[Vanishing \(\partial_{2}\cW_{\eta}\) at negative infinity]\label{lem:W_x_vanishing}
Let \cref{ass:input_datum} hold. Then, we have for \(q\in \sTV(\R)\cap \sC^{1}(\R)\) with \(\cW_{\eta}\) the nonlocal operator as in \cref{eq:relation_V_W}
\[
\lim_{x\rightarrow-\infty} \partial_{x}\cW_{\eta}[V(q)](x)=0.
\]
\end{lemma}
\begin{proof}
We have for all \(x\in\R\) and \(y^{*}\in\R_{\geq x}\)
\begin{align*}
    \big|\partial_{x}\cW_{\eta}[V(q)](x)\big|&= \big|\cW_{\eta}\big[V'(q) q'\big](x)\big|\\
    &\leq \tfrac{1}{\eta}\int_{x}^{y^{*}}\e^{\frac{x-y}{\eta}}|V'(q(y))||q'(y)|\dd y+\tfrac{1}{\eta}\int_{y^{*}}^{\infty} \e^{\frac{x-y}{\eta}}|V'(q(y))||q'(y)|\dd y\\
    &\leq \tfrac{1}{\eta}\int_{-\infty}^{y^{*}}|V'(q(y))||q'(y)|\dd y+\tfrac{1}{\eta}\int_{y^{*}}^{\infty} \e^{\frac{x-y}{\eta}}|V'(q(y))||q'(y)|\dd y\\
    &\leq \tfrac{1}{\eta}\|V'\|_{\sL^{\infty}((0,\|q\|_{\sL^{\infty}(\R)}))}|q|_{\sTV(-\infty,y^{*})}+\tfrac{1}{\eta}\int_{y^{*}}^{\infty} \e^{\frac{x-y}{\eta}}|V'(q(y))||q'(y)|\dd y.
\end{align*}
We can chose \(y^{*}\in\R_{<0}\) negative enough so that \(|q|_{\sTV(-\infty,y^{*})}\) is arbitrary small and letting \(x\rightarrow-\infty\) in the second term, we have by the monotone convergence that this term vanishes. Altogether, we obtain
\[
\lim_{x\rightarrow -\infty}|\partial_{x}\cW_{\eta}[V(q_{\eta})](x)|=0.
\]
\end{proof}
The next proposition establishes uniform \(\sTV\) bounds on the nonlocal velocity as long as the velocity function \(V\) satisfies a specific growth condition:
\begin{proposition}[Uniform (in \(\eta\)) total variation estimate of {\(\cW_{\eta}[V(q_{\eta})]\)}]\label{prop:uniform_TV_estimate}
Let \(\cW_{\eta}[V(q_{\eta})]\) as in \cref{lem:dynamics_W} be given, and assume that the velocity function satisfies
\begin{equation}
V'(x)x-V(x)+V\Big(\essinf_{y\in\R} q_{0}(y)\Big)\leq 0\quad \forall x\in \Big(\inf_{y\in\R} q_{0}(y),\|q_{0}\|_{\sL^{\infty}(\R)}\Big).\label{eq:velocity_TV_bound}
\end{equation}
Then, it holds \(\forall t\in[0,T]\)
\[
|\cW_{\eta}[V(q_{\eta})](t,\cdot)|_{\sTV(\R)}\leq |\cW_{\eta}[V(q_{\eta})](0,\cdot)]|_{\sTV(\R)}\leq \|V'\|_{\sL^{\infty}((0,\|q_{0}\|_{\sL^{\infty}(\R)}))}|q_{0}|_{\sTV(\R)}.
\]
\end{proposition}
\begin{proof}
Smoothing \(q_{\eta}\) as outlined in \cref{lem:stability} by using a standard mollifier \(\{\phi_{\eps}\}_{\eps\in\R_{>0}}\subseteq\sC^{\infty}(\R)\) as \(q_{0}^{\eps}\), we have \(\cW_{\eta}[V(q_{\eta}^{\eps})]\) smooth as well so that we can differentiate with regard to \(x\in\R\) in \cref{eq:W} and -- leaving out the dependency of \(\cW,q\) with regard to \(\eta\) and \(\eps\) and writing
\(\cW\) instead of \(\cW_{\eta}[V(q_{\eta})]\) -- we have for \((t,z)\in\OT\)
\begin{align}
&\partial_{t} \partial_{z}\cW(t,z)\label{eq:lukas}\\
&=\tfrac{1}{\eta} V'(q(t,z))\partial_{z}\cW(t,z)q(t,z)-\tfrac{1}{\eta^{2}}\int_{z}^{\infty} V'(q(t,x))\e^{\frac{z-x}{\eta}}\partial_{x}\cW(t,x) q(t,x) \dd x\notag\\
    &\quad +\tfrac{1}{\eta^{2}}V(q(t,z))^{2}-\tfrac{1}{\eta^{3}}\int_{z}^{\infty}V(q(t,x))^{2}\e^{\frac{z-x}{\eta}}\dd x\notag\\
    &\quad +\tfrac{2}{\eta}\cW(t,z)\partial_{z}\cW(t,z) -\big(\partial_{z}\cW(t,z)\big)^{2}-\cW(t,z)\partial_{z}^{2}\cW(t,z)\notag
    \intertext{using the identity \(\eta\partial_{x}\cW(t,x)=\cW(t,x)-V(q(t,x))\) derived from \cref{eq:relation_V_W} and established in \cref{eq:W} and an integration by parts}
    &=\tfrac{1}{\eta} V'(q(t,z))\partial_{z}\cW(t,z)q(t,z)-\tfrac{1}{\eta^{2}}\int_{z}^{\infty} V'(q(t,x))\e^{\frac{z-x}{\eta}}\partial_{x}\cW(t,x) q(t,x) \dd x\notag\\
    &\quad -\tfrac{2}{\eta^{2}}\int_{z}^{\infty}(\cW(t,x)-\eta\partial_{x}\cW(t,x))\big(\partial_{x}\cW(t,x)-\eta\partial_{x}^{2}\cW(t,x)\big)\e^{\frac{z-x}{\eta}}\dd x\notag\\
    &\quad +\tfrac{2}{\eta}\cW(t,z)\partial_{z}\cW(t,z) -\big(\partial_{z}\cW(t,z)\big)^{2}-\cW(t,z)\partial_{z}^{2}\cW(t,z)\notag\\
    &=\tfrac{1}{\eta} V'(q(t,z))\partial_{z}\cW(t,z)q(t,z)-\tfrac{1}{\eta^{2}}\int_{z}^{\infty} V'(q(t,x))\e^{\frac{z-x}{\eta}}\partial_{x}\cW(t,x) q(t,x) \dd x\notag\\
    &\quad -\tfrac{2}{\eta^{2}}\int_{z}^{\infty}\!\!\!\!\Big(\cW(t,x)\partial_{x}\cW(t,x)-\eta  \cW(t,x)\partial_{x}^{2}\cW(t,x)-\tfrac{\eta}{2} \big(\partial_{x}\cW(t,x)\big)^2\Big)\e^{\frac{z-x}{\eta}}\dd x\notag\\
    &\quad +\tfrac{2}{\eta}\cW(t,z)\partial_{z}\cW(t,z) -\cW(t,z)\partial_{z}^{2}\cW(t,z)\notag
    \intertext{and another integration by parts in the ``middle'' term}
    &=\tfrac{1}{\eta} V'(q(t,z))\partial_{z}\cW(t,z)q(t,z)-\tfrac{1}{\eta^{2}}\int_{z}^{\infty} V'(q(t,x))\e^{\frac{z-x}{\eta}}\partial_{x}\cW(t,x) q(t,x) \dd x\notag\\
    &\quad -\tfrac{1}{\eta}\int_{z}^{\infty} \big(\partial_{x}\cW(t,x)\big)^{2}e^{\frac{z-x}{\eta}}\dd x -\cW(t,z)\partial_{z}^{2}\cW(t,z).\notag
\end{align}
Computing next the change of total variation, we take advantage of the previously derived identity and have for \(t\in[0,T]\)
\begin{align*}
&\partial_{t}\int_{\R}|\partial_{x}\cW(t,x)|\dd x\\
&\stackrel{\hspace{-.2cm}\eqref{eq:lukas}\hspace{-.2cm}}{=}\tfrac{1}{\eta}\int_{\R}\sgn(\partial_{x}\cW(t,x)) V'(q(t,x))\partial_{x}\cW(t,x) q(t,x)\dd x\notag\\
&\quad -\tfrac{1}{\eta^{2}}\int_{\R}\sgn(\partial_{x}\cW(t,x))\int_{x}^{\infty} V'(q(t,y))\e^{\frac{x-y}{\eta}}\partial_{y}\cW(t,y) q(t,y) \dd y\dd x\notag\\
&\quad -\tfrac{1}{\eta}\int_{\R}\sgn(\partial_{x}\cW(t,x))\int_{x}^{\infty}\big(\partial_{y}\cW(t,y)\big)^{2}\e^{\frac{x-y}{\eta}}\dd y\dd x\notag\\
&\quad -\int_{\R}\sgn(\partial_{x}\cW(t,x)) \cW(t,x)\partial_{x}^{2}\cW(t,x)\dd x\notag
\intertext{and an integration by parts in the last term yields}
&=\tfrac{1}{\eta}\int_{\R}\sgn(\partial_{x}\cW(t,x)) V'(q(t,x))\partial_{x}\cW(t,x) q(t,x)\dd x\notag\\
&\quad -\tfrac{1}{\eta^{2}}\int_{\R}\sgn(\partial_{x}\cW(t,x))\int_{x}^{\infty} V'(q(t,y))\e^{\frac{x-y}{\eta}}\partial_{y}\cW(t,y) q(t,y) \dd y\dd x\notag\\
&\quad -\tfrac{1}{\eta}\int_{\R}\sgn(\partial_{x}\cW(t,x))\int_{x}^{\infty}\big(\partial_{y}\cW(t,y)\big)^{2}\e^{\frac{x-y}{\eta}}\dd y\dd x \notag\\
&\quad +\int_{\R}\sgn(\partial_{x}\cW(t,x)) \big(\partial_{x}\cW(t,x)\big)^{2}\dd x\notag\\
&\quad +\int_{\R}\delta(\partial_{x}\cW(t,x))\cW(t,x)\partial_{x}\cW(t,x)\partial_{x}^{2}\cW(t,x)\dd x\dd t\notag\\
&\quad +\lim_{x\rightarrow-\infty}|\partial_{x}\cW(t,x)|\cW(t,x)\underbrace{-\lim_{x\rightarrow\infty}|\partial_{x}\cW(t,x)|\cW(t,x)}_{\leq 0}
\intertext{sorting terms and recalling that \(\cW\geq 0\)}
&\leq\tfrac{1}{\eta}\int_{\R}|\partial_{x}\cW(t,x)|\big(V'(q(t,x))q(t,x)+\eta\partial_{x}\cW(t,x)\big)\dd x\notag\\
&\quad -\tfrac{1}{\eta^{2}}\int_{\R}\sgn(\partial_{x}\cW(t,x))\int_{x}^{\infty}\partial_{y}\cW(t,y)\e^{\frac{x-y}{\eta}}\big(V'(q(t,y))q(t,y)+\eta \partial_{y}\cW(t,y)\big)\dd y\dd x\notag\\
&\quad +\lim_{x\rightarrow-\infty}|\partial_{x}\cW(t,x)|\cW(t,x)
\intertext{using \cref{lem:W_x_vanishing} and a change of order of integration}
&=\tfrac{1}{\eta}\int_{\R}|\partial_{x}\cW(t,x)|\big(V'(q(t,x))q(t,x)\!+\!\cW(t,x)\!-\!V(q(t,x))\big)\dd x\\
&\ +\tfrac{1}{\eta^{2}}\!\!\!\int_{\R}\!\!\partial_{y}\cW(t,y)\big(V'(q(t,y))q(t,y)+\cW(t,y)-V(q(t,y))\big)\!\!\!\int_{-\infty}^{y}\!\!\!\!\!\!\!\e^{\frac{x-y}{\eta}}\sgn(\partial_{x}\cW(t,x))\dd x\dd y\\
\intertext{and recognizing that \(V'(q(t,x))q(t,x)+\cW(t,x)-V(q(t,x))\leq 0\ \forall (t,x)\in\OT\) thanks to  \cref{eq:velocity_TV_bound} and  \(\cW(t,x)\leq V\big(\essinf_{y\in\R}q_{0}(y)\big)\) }
&\leq \tfrac{1}{\eta}\int_{\R}|\partial_{x}\cW(t,x)|\big(V'(q(t,x))q(t,x)+\cW(t,x)-V(q(t,x))\big)\dd x\\
&\quad -\tfrac{1}{\eta^{2}}\int_{\R}|\partial_{y}\cW(t,y)|\big(V'(q(t,y))q(t,y)+\cW(t,y)-V(q(t,y))\big)\int_{-\infty}^{y}\e^{\frac{x-y}{\eta}}\dd x\dd y=0.
\end{align*}
Recalling that the involved functions originally depended on \(\eps,\eta\), we have by the previous inequality
\begin{align*}
    |\cW_{\eta}[V(q_{\eta}^{\eps})](t,\cdot)|_{\sTV(\R)}&\leq |\cW_{\eta}[V(q_{\eta}^{\eps})](0,\cdot)|_{\sTV(\R)}\\
    &\leq |V(q_{0}^{\eps}(\cdot))|_{\sTV(\R)}\leq \|V'\|_{\sL^{\infty}((0,\|q_{0}{\eps}\|_{\sL^{\infty}(\R)}))}|q_{0}^{\eps}|_{\sTV(\R)}\\
    &\leq|V(q_{0}^{\eps}(\cdot))|_{\sTV(\R)}\leq \|V'\|_{\sL^{\infty}((0,\|q_{0}\|_{\sL^{\infty}(\R)}))}|q_{0}|_{\sTV(\R)}
\end{align*}
by the typical approximation results on \(q_{0}\) by \(q_{0}^{\eps}\).
However, this is indeed the postulated estimate, uniform in \(\eps\) and \(\eta\).
\end{proof}
\begin{remark}[The meaning of \cref{eq:velocity_TV_bound}]
Making \cref{eq:velocity_TV_bound} slightly more restrictive by postulating (\(V\) is monotonically decreasing)
\[
V'(x)x-V(x)+V(0)\leq 0
\]
we can use Taylor polynomials with remainder so that we can write
\[ V'(x)x-V(x)+V(0)=\int_0^x s V''(s)\dd s.\]
This expression makes it easier to classify functions which fulfill \cref{eq:velocity_TV_bound}. On the one hand convex \(V\) do not satisfy \cref{eq:velocity_TV_bound}, but on the other hand every concave velocity functions are well-suited. In particular, this includes the commonly used class of velocity functions introduced by Greenshields in \cite{greenshields}, i.e.,
\begin{equation}
V(x)=V_{\max}\Big(1-\big(\tfrac{x}{q_{\max}}\big)^k \Big),
\label{eq:greenshields}\end{equation}
for some $k\in \mathbb{N}$ and positive constants $V_{\max}\in\R_{>0},\ q_{\max}\in\R_{>0}$ which represent maximum velocity and density.
\end{remark}
\subsection{Total variation bounded velocity implies the entropy admissibility}\label{sec:entropy_convergence}
In this section, we show that the weak solution of the nonlocal dynamics converges to the entropy solution given that the nonlocal approximation of the velocity converges.
\begin{theorem}[Convergence to the local entropy solution]\label{theo:convergenceentropy}
For the problem given in \cref{defi:problem_setup} assume that it holds
\begin{align}
-\infty<\tfrac{V'(s)}{s}&<0\qquad  \forall s\in \Big[\inf_{x\in\R}q_{0}(x),\|q_{0}\|_{\sL^{\infty}(\R)}\Big]\label{eq:lower_bound_V_prime}\\
x\mapsto xV(x) &\text{ strictly convex/concave on \(\Big[\essinf_{y\in\R}q_{0}(y),\|q_{0}\|_{\sL^{\infty}(\R)}\)}\Big]\label{eq:velocity_strictly_convex}\\
\exists C\in\R_{\geq0}:\ & \sup_{\eta\in\R_{>0}}\big|\cW_{\eta}[V(q_{\eta})]\big|_{\sL^{\infty}((0,T);\sTV(\R))}\leq C.\label{eq:TV_uniform}
\end{align}
Then, the weak solution \(q_{\eta}\in\sC\big([0,T];\sL^{1}_{\,\loc}(\R)\big)\) of \cref{defi:problem_setup}  converges weakly star to the local entropy solution \(q^{*}\in \sC\big([0,T];\sL^{1}_{\,\loc}(\R)\big)\) as in \cref{defi:entropy_solution}, i.e.,
\[\forall K\overset{\text{c}}{\subset}\R\text{ compact } \forall g\in \sL^{1}\big((0,T);\sL^{1}(K)\big):
\lim_{\eta\rightarrow 0}\int_{0}^{T}\!\!\!\!\int_{K}\!\!\big(q_{\eta}(t,x)-q^{*}(t,x)\big)g(t,x)\dd x\dd t=0
\]
\end{theorem}
\begin{proof}
Recalling \cref{defi:entropy_solution} and particularly \cref{ineq:Entropy_inequality} we need to show that
\[
\lim_{\eta\rightarrow 0}\mEF[\phi,\alpha,q_{\eta}]\geq 0 \ \forall \phi\in C^{1}_{\text{c}}(\OT)
\]
according to \cref{theo:one_strict_entropy} and the assumption in \cref{eq:velocity_strictly_convex} -- for one strictly convex entropy \(\alpha\).
To this end, we again assume that \(q_{\eta}^{\eps}\) for \(\eps\in\R_{>0}\) is smooth by taking advantage of \cref{lem:stability} and having indeed only smoothed the initial datum \(q_{0}\) by \(q_{0}^{\eps}\).
Then, we can go into \cref{ineq:Entropy_inequality} and manipulate further to arrive for \(\alpha\in \sW^{2,\infty}_{\loc}(\R),\phi\in\sC_{\text{c}}^{1}((-42,T)\times\R;\R_{\geq0})\) arbitrary but given at (suppressing the dependency on \(\eps\in\R_{>0}\))
\begin{align*}
    &\mEF[\phi,\alpha,q_{\eta}]\\
    &=-\iint_{\OT}\Big(\alpha'(q_{\eta}(t,x))\partial_{t}q_{\eta}(t,x)+\beta'(q_{\eta}(t,x))\partial_{x}q_{\eta}(t,x)\Big)\phi(t,x)\dd x\dd t\\
    &=-\iint_{\OT}\alpha'(q_{\eta}(t,x))\Big(\partial_{t}q_{\eta}(t,x)+\partial_{x}\big(V(q_{\eta}(t,x))q_{\eta}(t,x)\big)\Big)\phi(t,x)\dd x\dd t
    &\intertext{inserting the strong form of the nonlocal equation, i.e., the equation in \cref{defi:problem_setup}}
&=-\iint_{\OT}\alpha'(q_{\eta}(t,x))\partial_{x}q_{\eta}(t,x)\big(V(q_{\eta}(t,x))-\cW_{\eta}[V(q_{\eta})](t,x)\big)\phi(t,x)\dd x\dd t\\
&\quad -\iint_{\OT}\alpha'(q_{\eta}(t,x))q_{\eta}(t,x)\partial_{x}\Big(V(q_{\eta}(t,x))-\cW_{\eta}[V(q_{\eta})](t,x)\Big)\phi(t,x)\dd x\dd t\notag
\intertext{and with another integration by parts in the second term}
&=\iint_{\OT}\alpha''(q_{\eta}(t,x))q_{\eta}(t,x)\partial_{x}q_{\eta}(t,x)\Big(V(q_{\eta}(t,x))-\cW_{\eta}[V(q_{\eta})](t,x)\Big)\phi(t,x)\dd x\dd t\\
&\quad +\iint_{\OT}\alpha'(q_{\eta}(t,x))q_{\eta}(t,x)\big(V(q_{\eta}(t,x))-\cW_{\eta}[V(q_{\eta})](t,x)\big)\partial_x \phi(t,x)\dd x\dd t.\notag\\
\intertext{Setting as convex entropy \(\alpha''(x)=-\tfrac{V'(x)}{x}>0\ \forall x\in \big(\inf_{s\in\R}q_{0}(s),\|q_{0}\|_{\sL^{\infty}(\R)}\big)\) which is guaranteed by \cref{eq:lower_bound_V_prime} }
&=-\iint_{\OT} V'(q_{\eta}(t,x))V(q_{\eta}(t,x))\partial_{x}q_{\eta}(t,x)\phi(t,x)\dd x\dd t \\
&\quad +\iint_{\OT}V'(q_{\eta}(t,x))\partial_{x}q_{\eta}(t,x)\cW_{\eta}[V(q_{\eta})](t,x)\phi(t,x)\dd x\dd t\\
&\quad +\iint_{\OT}\alpha'(q_{\eta}(t,x))q_{\eta}(t,x)\big(V(q_{\eta}(t,x))-\cW_{\eta}[V(q_{\eta})](t,x)\big)\partial_x \phi(t,x)\dd x\dd t\notag
\intertext{Recalling \(V'(q_{\eta}(t,x))\partial_{x}q_{\eta}(t,x)=\partial_{x}\cW_{\eta}[V(q_{\eta})](t,x)-\eta\partial_{x}^{2}\cW_{\eta}[V(q_{\eta})](t,x)\ \forall (t,x)\in\OT\) which follows from \cref{eq:relation_V_W} when differentiating two times in space we continue}
&=-\tfrac{1}{2}\iint_{\OT}\partial_{x}\big(V(q_{\eta}(t,x))^{2}\big)\phi(t,x)\dd x\dd t\\
&\quad+\iint_{\OT}\partial_{x}\cW_{\eta}[V(q_{\eta})](t,x)\cW_{\eta}[V(q_{\eta})](t,x)\phi(t,x)\dd x\dd t\\
&\quad -\eta\iint_{\OT}\partial_{x}^{2}\cW_{\eta}[V(q_{\eta})](t,x)\cW_{\eta}[V(q_{\eta})](t,x)\phi(t,x)\dd x\dd t\\
&\quad +\iint_{\OT}\alpha'(q_{\eta}(t,x))q_{\eta}(t,x)\big(V(q_{\eta}(t,x))-\cW_{\eta}[V(q_{\eta})](t,x)\big)\partial_{x}\phi(t,x)\dd x\dd t\notag\\
\intertext{and another integration by parts on the middle term yields}
&=\tfrac{1}{2}\iint_{\OT}\big(V(q_{\eta}(t,x))^{2}\big)\partial_{x}\phi(t,x)\dd x\dd t +\tfrac{1}{2}\iint_{\OT}\partial_{x}\big(\cW_{\eta}[V(q_{\eta})](t,x)\big)^{2}\phi(t,x)\dd x\dd t\\
&\quad+\eta\iint_{\OT}\partial_{x}\cW_{\eta}[V(q_{\eta})](t,x)\cW_{\eta}[V(q_{\eta})](t,x)\partial_{x}\phi(t,x)\dd x\dd t\\
&\quad+\underbrace{\eta\iint_{\OT}\big(\partial_{x}\cW_{\eta}[V(q_{\eta})](t,x)\big)^{2}\phi(t,x)\dd x\dd t}_{\geq 0}\\
&\quad +\iint_{\OT}\alpha'(q_{\eta}(t,x))q_{\eta}(t,x)\big(V(q_{\eta}(t,x))-\cW_{\eta}[V(q_{\eta})](t,x)\big)\partial_x \phi(t,x)\dd x\dd t\notag
\intertext{and with \(I\:(\essinf_{x\in\R}q_{0}(x),\|q_{0}\|_{\sL^{\infty}(\R)})\)}
&\geq\tfrac{1}{2}\iint_{\OT}\big(V(q_{\eta}(t,x))\big)^{2}\partial_{x}\phi(t,x)\dd x\dd t -\tfrac{1}{2}\iint_{\OT}\big(\cW_{\eta}[V(q_{\eta})](t,x)\big)^{2}\partial_{x}\phi(t,x)\dd x\dd t\\
&\quad-\eta T\sup_{t\in[0,T]}|\cW_{\eta}[V(q_{\eta})](t,\cdot)| _{\sTV(\R)}\|\cW_{\eta}[V(q_{\eta})]\|_{\sL^{\infty}((0,T);\sL^{\infty}(\R))}\|\partial_{2}\phi\|_{\sL^{\infty}(\OT)}\\
&\quad -\|\alpha'\|_{\sL^{\infty}(I)}\|q_{0}\|_{\sL^{\infty}(\R)}T\|V(q_{\eta})-\cW_{\eta}[V(q_{\eta})]\|_{\sC([0,T];\sL^{1}(\R))}\|\partial_{2}\phi\|_{\sL^{\infty}(\OT)}\\
&\geq -\|V\|_{\sL^{\infty}(I)}\|\partial_{2}\phi\|_{\sL^{\infty}(\OT)} T\|V(q_{\eta})-\cW_{\eta}[V(q_{\eta})]\|_{\sC([0,T];\sL^{1}(\R))}\\
&\quad-\eta T\sup_{t\in[0,T]}|\cW_{\eta}[V(q_{\eta})](t,\cdot)| _{\sTV(\R)}\|\cW_{\eta}[V(q_{\eta})]\|_{\sL^{\infty}((0,T);\sL^{\infty}(\R))}\|\partial_{2}\phi\|_{\sL^{\infty}(\OT)}\\
&\quad -\|\alpha'\|_{\sL^{\infty}(I)}\|q_{0}\|_{\sL^{\infty}(\R)}T\|V(q_{\eta})-\cW_{\eta}[V(q_{\eta})]\|_{\sC([0,T];\sL^{1}(\R))}\|\partial_{2}\phi\|_{\sL^{\infty}(\OT)}\overset{\eta\rightarrow 0}=0
\end{align*}
when recalling that by \cref{eq:TV_uniform} together with \cref{eq:relation_V_W}, we have
\begin{align*}
\partial_{x}\cW_{\eta}[V(q_{\eta})]&\equiv\tfrac{1}{\eta}\big(\cW_{\eta}[V(q_{\eta})]-V(q_{\eta})\big)\Longrightarrow \eta\partial_{x}\cW_{\eta}[V(q_{\eta})]\equiv \cW_{\eta}[V(q_{\eta})]-V(q_{\eta})\\
&\Longrightarrow \eta|\partial_{x}\cW_{\eta}[V(q_{\eta})]|_{\sL^{\infty}((0,T);\sTV(\R))}=\|\cW_{\eta}[V(q_{\eta})]-V(q_{\eta})\|_{\sC([0,T];\sL^{1}(\R))}.
\end{align*}
Thus, we have that the right hand side of the previous term converges to zero which means that -- applied to the previous estimate the first and second term clearly converge to zero if \(\eta\rightarrow 0\) while the second term converges to zero by the uniform \(\sTV\) bound of \(\cW_{\eta}[V(q_{\eta})]\) and the factor \(\eta\) in front of it.
Recalling that we have the dependency on the smoothing parameter \(\eps\in\R_{>0}\) it is enough to notice that \(\mEF[\phi,\alpha,q]\) is continuous in \(\sC\big([0,T];\sL^{1}_{\loc}(\R)\big)\) with regard to \(q\) so that we can let \(\eps\rightarrow 0\) in the previous estimate using \cref{lem:stability} to obtain the convergence for the non-smoothed version.

As \(q_{\eta}\) is essentially bounded as stated in \cref{theo:maximum_principle} as maximum principle, we know by the weak star compactness that there exists
\[
q^{*}\!\in\sL^{\infty}((0,T);\sL^{\infty}(\R)),\ \{\eta_{k}\}_{k\in\N}\subset \R_{>0}: \lim_{k\rightarrow\infty} \eta_{k}=0: q_{\eta_{k}}\!\!\weakstar q^{*} \text{ in } \sL^{\infty}((0,T);\sL^{\infty}(\R)).
\]
The convergence up to now was only on subsequences, however, as every subsequence converges to the local entropy solution and as this entropy solution is unique as guaranteed in \cref{theo:local_conservation_existence_uniqueness} and \cref{theo:one_strict_entropy}, we obtain the convergence for all sequences.
\end{proof}

\begin{remark}[Strong convergence of the solution and more]\label{rem:strong_convergence}
Interestingly, in the previous proof the typical compactness estimates \cite[Theorem 4.1]{coclite2020general}, \cite[Proposition 4.1, Lemma 4.2]{keimer43} are not necessary to obtain the convergence against the local entropy solution. This can be explained by the fact that we obtain instead the compactness in the velocity \(\cW\) -- the only nonlinear part in the considered conservation law -- although this compactness is not explicitly used. On the other hand, one only obtains the weak star convergence against the entropy solution. The convergence in \(\sC([0,T];\sL^{1}_{\loc}(\R))\) indeed requires explicit compactness arguments, which carry over from the compactness of the nonlocal velocity whenever we assume that \(V\) is a diffeomorphism. Thus, assuming for instance that \(V'(s)<0\ \forall s\in \Big[\essinf_{y\in\R}q_{0}(y),\|q_{0}\|_{\sL^{\infty}(\R)}\Big]\), we obtain analogously to \cite[Theorem 4.1]{coclite2020general} the strong convergence. We do not detail it further, but note it in our general convergence result in the next section.
\end{remark}
\subsection{The singular limit problem: Convergence of solutions to nonlocal conservation laws to the local entropy solution}
In this section, we will present the convergence result which is mainly a collection of the previous results obtained in \cref{sec:uniform_TV} and \cref{sec:entropy_convergence}. However, as we do not deal with a classical compactness estimate in \(\sBV\) but instead in \(\sTV\cap \sL^{\infty}\) it is worth detailing the convergence implied by the uniform \(\sTV\) bound in \cref{sec:uniform_TV}:
\begin{theorem}[Convergence to the entropy solution]\label{theo:main}
Let \cref{ass:input_datum} hold and in addition
\begin{align}
-\infty<\tfrac{V'(s)}{s}&<0 && \forall s\in \Big[\essinf_{x\in\R}q_{0}(x),\|q_{0}\|_{\sL^{\infty}(\R)}\Big]   \label{eq:V_prime} \\
s\mapsto sV(s) \text{ strictly} &\text{ convex/concave } &&\forall s\in\Big[\essinf_{y\in\R}q_{0}(y),\|q_{0}\|_{\sL^{\infty}(\R)}\Big] \label{eq:V_prime_1}\\
V'(s)s+V\Big(\essinf_{y\in\R}q_{0}(y)\Big)&\leq V(s)&& \forall s\in \Big[\essinf_{x\in\R} q_{0}(x),\|q_{0}\|_{\sL^{\infty}(\R)}\Big].\label{eq:V_prime_2}
\end{align}
Then, the solution \(q_{\eta}\in\sC\big([0,T];\sL^{1}_{\loc}(\R)\big)\) of the nonlocal conservation law \cref{defi:problem_setup} converges to the local entropy solution \(q^{*}\in\sC\big([0,T];\sL^{1}_{\loc}(\R)\big)\) of \cref{defi:problem_setup_local} for \(\eta\rightarrow 0\) and it holds
\begin{align*}
\lim_{\eta\rightarrow 0} \|\cW_{\eta}[V(q_{\eta})]-V(q^{*})\|_{\sC\left([0,T];\sL^{1}_{\loc}(\R)\right)}&=0\ \wedge\ q_{\eta}\weakstar q^{*}\in \sL^{\infty}((0,T);\sL^{\infty}(\R)).
\end{align*}
If \(V'(s)<0\ \forall s\in\big[\essinf_{x\in\R} q_{0}(x),\|q_{0}\|_{\sL^{\infty}(\R)}\big]\), we obtain strong convergence in \(\sC\big([0,T];\sL^{1}_{\loc}(\R)\big)\), i.e.
\[
\lim_{\eta\rightarrow 0}\|q_{\eta}-q^{*}\|_{\sC([0,T];\sL^{1}_{\loc}(\R))}=0.
\]
\end{theorem}
\begin{proof}
The proof is a combination of the uniform \(\sTV\) bound in \cref{prop:uniform_TV_estimate}, the entropy admissibility in \cref{theo:convergenceentropy} and \cref{rem:strong_convergence}.
\end{proof}
\begin{remark}[Meaning of the requirements]
Let us shortly discuss the assumptions with respect to their applicability. The condition \cref{eq:V_prime} does not allow zero initial density as long as \(V'(s)\in {\scriptstyle \mathcal{O}}(s)\) for \(s\rightarrow 0\) which is rather restrictive. The upper bound, however, not as much as it just excludes cases where \(V\) gets locally constant or its derivative is zero at \(0\).

\cref{eq:V_prime_1} is the classical assumption for (local) conservation laws and is satisfied by a variety of velocity functions, in particular also by the named \cref{eq:greenshields} velocity. Finally, \cref{eq:V_prime_1} is satisfied by \cref{eq:greenshields} as well.

Concluding, all assumptions seem to be rather natural and not very restrictive in the context of conservation laws except for the zero density case.

This case, however, is an immediate consequence of requiring a strictly convex entropy in the proof of \cref{theo:convergenceentropy}, namely that \(x\mapsto -\frac{V'(x)}{x}\) is strictly positive. As can be seen, for \(x\rightarrow 0\) this term diverges to \(\infty\) and is thus not covered by \cref{theo:one_strict_entropy} that one entropy is enough for specific velocities to obtain the uniqueness of solutions to local conservation laws. It might be possible that this result can be extended to entropies which are unbounded at the boundary, however, for now, \cref{theo:one_strict_entropy} remains the restricting point.

An interesting case is the choice of a linear velocity, say
\(V(s)=1-s,\ s\in \R\) and \(q_0 \geqq \eps > 0\). Then, clearly, \cref{eq:V_prime} is satisfied and so is \cref{eq:V_prime_1} and \cref{eq:V_prime_2}.
This is -- except for the restriction on the initial datum --  in line with \cite{coclite2020general} where  a general convergence result for \textbf{nonlocal in the solution} is established as in the case of a linear velocity \textbf{nonlocal in velocity} and \textbf{nonlocal in the solution} coincide (compare also \cref{eq:local_nonlocal}).

Differently put, our analysis was fine enough to cover the convergence result in the coinciding case with the nonlocal in the solution conservation law and the singular limit problem when restricting the initial so that it is uniformly positive.

Another interesting case is the choice of a quadratic velocity, say \(V(s)=1-s^{2},\ s\in\R\). Then, \cref{eq:V_prime} is satisfied without restricting the initial datum, \cref{eq:V_prime_1} is also satisfied and \cref{eq:V_prime_2} results in
\[
-2s^{2}+1\leq 1-s^{2}\Longleftrightarrow -s^{2}\leq 0,\  s\in[0,1]
\]
and which is true for all \(s\in[0,1]\). Altogether, the convergence to the local entropy solution holds for any initial datum in the case of the named quadratic velocity function.
\end{remark}

\section{Convergence results for arbitrary kernels and monotone initial datum}\label{sec:convergence_monotone}
As the obtained convergence in \cref{theo:main} does not cover all cases due to the conditions \crefrange{eq:V_prime}{eq:V_prime_2} and even more we have obtained the convergence only for the exponential kernel (although it might be possible to broaden it, compare \cite{keimer42}), we look in this chapter into whether conservation laws with nonlocal velocity are monotonicity preserving which would give us for these monotone initial data the convergence for general kernels.
And indeed, similar to \cite[Thm.\ 4.18]{pflug4} we get the following result
\begin{theorem}[Monotonicity of the nonlocal solution]\label{theo:monotonicity}
Assume that \cref{ass:input_datum} holds, that \(q_0\in\sL^{\infty}(\R;\R_{\geq0})\cap \sTV(\R)\) is monotone
and that \(V\in\sC^{3}(\R)\).
Then, the solution $q_{\eta}$ of the conservation law with nonlocal velocity as in \cref{defi:problem_setup} is monotonicity preserving on the entire time horizon.
\end{theorem}
\begin{proof}
We start with monotone increasing data and smooth the initial datum with a standard mollifier \(\{\phi_{\eps}\}_{\eps\in\R_{>0}}\subset \sC^{\infty}(\R;\R_{\geq0})\) as in \cite[Remark C.18]{leoni}. Then, the smoothed initial datum \(q_{0}^{\eps}\in\sC^{\infty}(\R)\) is still monotonically increasing and according to \cref{lem:stability} and \(V\in\sC^{3}(\R)\) the smoothed solution \(q^{\eps}_{\eta}\) is twice continuously differentiable.
Now, we prove for the smoothed solution that the monotonicity is preserved over time. To this end, compute the time derivative of the spatial derivative for \((t,x)\in\OT\)
\begin{align*}
    \partial_t\partial_x q_{\eta}^{\eps}(t,x)&\overset{\eqref{eq:nonlocal_dynamics}}{=}-\partial_x^{2}\left( q_{\eta}^{\eps}(t,x) \cW_{\eta}[V(q_{\eta}^{\eps}),\gamma](t,x)\right)\\
    &=-\partial_{x}^{2}q_{\eta}^{\eps}(t,x) \cW_{\eta}[V(q_{\eta}^{\eps}),\gamma](t,x)-2 \partial_{x}q_{\eta}^{\eps}(t,x) \partial_x \cW_{\eta}[V(q_{\eta}^{\eps}),\gamma](t,x)\\
    &\qquad-\partial_{x}^{2} \cW_{\eta}[V(q_{\eta}^{\eps}),\gamma](t,x) q_{\eta}^{\eps}(t,x).
\end{align*}
Now, for \(t\in[0,T]\) let $\tilde{x}\in \R$ be a minimal point of $\partial_{x}q_{\eta}^{\eps}(t,x)$,\ i.e., \[\tilde{x}\in\argmin \partial_{x}q_{\eta}^{\eps}(t,x).\] Then, we also have that \(\partial_{x}^{2}q_{\eta}^{\eps}(t,x)\big|_{x=\tilde{x}}=0\) and assume that \(\partial_{2}q_{\eta}^{\eps}(t,\tilde{x})=0\) so that we indeed consider the first time where the monotonicity might be destroyed, we obtain
\begin{align}
    \partial_t(\partial_x q_{\eta}^{\eps}(t,x))\big|_{x=\tilde{x}}&=-\partial_{2}^{2} \cW_{\eta}[V(q_{\eta}^{\eps}),\gamma](t,\tilde{x}) q_{\eta}^{\eps}(t,\tilde{x}).\label{eq:4142}
\end{align}
Detailing the second order derivative of the nonlocal term, we obtain
\begin{align*}
   \partial^2_{2} \cW_{\eta}[V(q_{\eta}^{\eps}),\gamma](t,x)&=\partial_{x}^{2}\tfrac{1}{\eta}\int_{x}^{\infty}\gamma\big(\tfrac{y-x}{\eta}\big)V(q_{\eta}^{\eps}(t,y))\dd y\\
   &=\partial_{x}\Big(\tfrac{1}{\eta}\int_{x}^{\infty}\gamma\big(\tfrac{y-x}{\eta}\big)V'(q_{\eta}^{\eps}(t,y))\partial_{2}q_{\eta}^{\eps}(t,y)\dd y\Big)\\
   &=-\tfrac{1}{\eta}\gamma(0)V'(q_{\eta}^{\eps}(t,x))\partial_{x}q_{\eta}^{\eps}(t,x)\\
   &\quad-\tfrac{1}{\eta^{2}}\int_{x}^{\infty}\gamma'\big(\tfrac{y-x}{\eta}\big)V'(q_{\eta}^{\eps}(t,y))\partial_{2}q_{\eta}^{\eps}(t,y)\dd y.
\end{align*}
Evaluating at \(\tilde{x}\in\R\), it yields
\begin{align*}
 \partial^2_{2} \cW_{\eta}[V(q_{\eta}^{\eps}),\gamma](t,\tilde{x})&=-\tfrac{1}{\eta}\gamma(0)V'(q_{\eta}^{\eps}(t,\tilde{x}))\partial_{2}q_{\eta}^{\eps}(t,\tilde{x})\\
 &\quad-\tfrac{1}{\eta^{2}}\int_{\tilde{x}}^{\infty}\gamma'\big(\tfrac{y-\tilde{x}}{\eta}\big)V'(q_{\eta}^{\eps}(t,y))\partial_{2}q_{\eta}^{\eps}(t,y)\dd y\\
 &=-\tfrac{1}{\eta^{2}}\int_{\tilde{x}}^{\infty}\gamma'\big(\tfrac{y-\tilde{x}}{\eta}\big)V'(q_{\eta}^{\eps}(t,y))\partial_{2}q_{\eta}^{\eps}(t,y)\dd y\leq 0
\end{align*}
as \(\gamma'\leqq 0\) as well as \(V'\leqq 0\).
Relating back to \cref{eq:4142} we thus obtain
\begin{align*}
      \partial_t(\partial_x q_{\eta}^{\eps}(t,x))\big|_{x=\tilde{x}}\geq 0.
\end{align*}
However, this means that the monotonicity is preserved. Letting \(\eps\rightarrow 0\), we end up with the claim.
In the case of monotone decreasing initial datum, the proof is similar with the proper adjustments on the monotonicity. We do not detail it here.
\end{proof}
The previous result shows the difference between nonlocality in the velocity and nonlocality in the solution, as indeed a similar result does not hold for the nonlocality in the solution.
\begin{remark}[Monotonicity preserving scheme without further restrictions on the velocity]\label{rem:monotonicity}
Note that in contrast to \cite[Theorem 4.13, Theorem 4.18]{pflug4} no restriction on the sign of the second derivative of $V$ needs to be made to obtain monotonicity preserving dynamics. This is relevant from an approximation point of view. If one would aim to define solutions of local conservation laws as limits for nonlocal conservation laws and if one might want to consider related optimal control and control problems in the nonlocal regime to approximate the corresponding local optimal control and control problems, the nonlocal in the velocity approximation might be superior as it preserves monotonicity (as the local equation as well) and is thus closer and more well-behaving as an approximation.
\end{remark}
The following Lemma gives us a uniform total variation bound assuming that the initial datum is monotone.
\begin{lemma}[Uniform \(\sTV\) bounds thanks to the monotonicity]\label{lem:uniform_TV_monotone}
Let \cref{ass:input_datum} hold and assume that the initial datum \(q_{0}\in\sL^{\infty}(\R)\) is either monotonically increasing or decreasing.. Then, the solution to the nonlocal conservation law \cref{defi:problem_setup} satisfies
\begin{align*}
    \forall (\ndt,t)\in \R_{>0}\times [0,T]: |q_{\ndt}(t,\cdot)|_{\sTV(\R)}\leq \norm{q_0}_{\sTV(\R)}.
\end{align*}
\end{lemma}
\begin{proof}
This is a direct consequence of the conservation of monotonicity as outlined in \cref{theo:monotonicity}.
\end{proof}
Having this uniform \(\sTV\) bound, we obtain by classical compactness arguments the strong convergence in \(\sL^{1}_{\loc}\). As we just have given the \(\sTV\) bounds on the spatial dependency we miss some time compactness to obtain the strong convergence in \(\sC([0,T];\sL^{1}(\R))\). However, this is very similar to \cite{coclite2020general} and we will thus directly state the convergence result.
\begin{theorem}[Convergence to the entropy solution]\label{theo:convergence_monotone} For the problem given in \cref{defi:problem_setup} let the input datum as in \cref{ass:input_datum} be given together with
\begin{equation}
\int_{\R_{\geq0}}\gamma(x)x\dd x<\infty.\label{eq:additional_kernel_integrability}
\end{equation}
Let the initial datum be monotone, then the weak solution \(q_{\eta} \in\sC\big([0,T];\sL^{1}_{\loc}(\R)\big)\cap \sL^{\infty}((0,T),\sL^{\infty}(\R))\) converges to the local entropy solution \(q^{*}\in\sC\big([0,T];\sL^{1}_{\loc}(\R)\big)\cap \sL^{\infty}((0,T),\sL^{\infty}(\R))\) as defined in \cref{defi:problem_setup_local} in the \(\sC\big([0,T];\sL^{1}_{\loc}(\R)\big)\) topology, i.e.
\[
\forall K\subset\R \text{ compact:}\ \lim_{\eta\rightarrow 0}\|q_{\eta}-q^{*}\|_{\sC\big([0,T];\sL^{1}(K)\big)}=0.
\]
\end{theorem}
\begin{proof}
Due to \cref{defi:entropy_solution} we need to show that
\[
\lim_{\eta\rightarrow 0}\mEF[\phi,\alpha,q_{\eta}]\geq 0 \ \forall \phi\in \sC^{1}_{\text{c}}(\OT)
\]
Thanks to \cref{lem:stability} we can approximate \(q_{\eta}\) by \(\{q_{\eta}^{\eps}\}_{\eps\in\R_{>0}}\subset\sC^{2}(\OT)\) and as it holds
\[
\lim_{\eps\rightarrow 0} \big|\mEF[\phi,\alpha,q_{\eta}^{\eps}]-\mEF[\phi,\alpha,q_{\eta}]\big|=0
\]
we can perform the following entropy manipulations with the smooth solutions.
Following the proof in \cref{theo:convergenceentropy} and recalling the entropy solution \cref{defi:entropy_solution} we have for \(\eps\in\R_{>0}\) under the assumption that \(q_{\eta}\) and thus also \(q_{\eta}^{\eps}\) are monotonically increasing
\begin{align}
    \mEF[\phi,\alpha,q_{\eta}^{\eps}] &=\iint_{\OT}\!\!\!\alpha''(q_{\eta}^{\eps})q_{\eta}^{\eps}\partial_{x}q_{\eta}^{\eps}\Big(V\big(q_{\eta}^{\eps}\big)-\cW_{\eta}[V(q_{\eta}^{\eps}),\gamma]\Big)\phi\dd x\dd t\label{eq:convergence_1}\\
&\quad +\iint_{\OT}\!\!\!\alpha'\big(q_{\eta}^{\eps}\big)q_{\eta}^{\eps}\Big(V\big(q_{\eta}^{\eps}\big)-\cW_{\eta}\big[V\big(q_{\eta}^{\eps}\big),\gamma\big]\Big)\phi(t,x)\dd x\dd t.\label{eq:convergence_2}
\end{align}
The first term \cref{eq:convergence_1} is nonnegative as \(\alpha''\geqq 0\leqq q_{\eta}^{\eps}\),\ \(\partial_{x}q_{\eta}^{\eps}\geqq 0\leqq \phi\) thanks to the monotonicity, and as \(V\) is monotonically decreasing it also holds
\[
V\big(q_{\eta}^{\eps}(t,x)\big)-\cW_{\eta}\big[V\big(q_{\eta}^{\eps}\big),\gamma\big](t,x)\geqq 0\quad \forall (t,x)\in\OT.
\]
A lower bound for the second term \cref{eq:convergence_2} is for \(I\coloneqq\big(\essinf_{x\in\R}q_{0}(x),\|q_{0}\|_{\sL^{\infty}(\R)}\big)\) and \(A\: \norm{\alpha'} _{\sL^\infty(I)}\norm{\phi_x}_{\sL^\infty(\supp(\phi))}\norm{q_0^{\eps}}_{\sL^\infty(\R)}\) given by
\begin{align}
 \eqref{eq:convergence_2}&\geq -A\int_0^T \int_\R
      \abs{\cW_{\eta}[V(q_{\eta}^{\eps}),\gamma](t,x)-V(q_{\eta}^{\eps}(t,x))} \dd x \dd t\\
      &=-A\int_0^T \int_\R
      V(q_{\eta}^{\eps}(t,x))-\tfrac{1}{\eta}\int_{x}^{\infty}\gamma\big(\tfrac{y-x}{\eta}\big)V(q_{\eta}^{\eps}(t,y))\dd y \dd x \dd t\\
      &=-A\int_0^T \int_\R
      V(q_{\eta}^{\eps}(t,x))-\int_{0}^{\infty}\gamma(z)V(q_{\eta}^{\eps}(t,z\eta+x))\dd z \dd x \dd t\\
      &=-A\int_0^T \int_{0}^{\infty}\gamma(z)\int_\R
      V(q_{\eta}^{\eps}(t,x))-V(q_{\eta}^{\eps}(t,z\eta+x))\dd x\dd z \dd t
      \intertext{and using \cref{lem:uniform_TV_monotone}}
      &\geq -AT\|V'\|_{\sL^{\infty}(I)}|q_{0}|_{\sTV(\R)} \eta\int_{0}^{\infty}\gamma(z)z\dd z\\
      &=-AT\|V'\|_{\sL^{\infty}(I)} \eta\int_{0}^{\infty}\gamma(z)z\dd z.\label{eq:42424242}
\end{align}
Thanks to \cref{eq:additional_kernel_integrability} \(A\) is bounded uniformly in \(\eps\) so that for \(\eta\rightarrow 0\) \cref{eq:42424242} converges to zero. Letting \(\eps\rightarrow 0\) yields the claim.
\end{proof}
\begin{remark}[The additional integrability on \(\gamma\)]
Note that the assumption on \(\gamma\) to be not only in \(\sL^{1}(\R_{\geq0};\R_{\geq0})\) monotonically decreasing (as postulated in \cref{ass:input_datum}) but also to satisfy \cref{eq:additional_kernel_integrability} is no restriction for kernels with finite support and the classically used exponential kernel.
However, for other kernels \(\gamma\) with \(\supp(\gamma)=\R_{\geq0}\) it states that these kernels need to decay faster than \(x\mapsto \tfrac{1}{x^{2}}\) for \(x\rightarrow\infty\).
\end{remark}

\begin{remark}[Oleinik type estimate]\label{rem:oleinik}
Note that the estimates before also give the Oleinik entropy condition \cite{oleinik_english,oleinik} which together with the strong convergence in \(\sL^{1}_{\loc}\) yields the uniqueness of weak (local) solutions as long as the flux \(x\mapsto xV(x)\) is strictly concave/convex, i.e., \(\forall x\in \Big(\essinf_{x\in\R}q_{0}(x),\|q_{0}\|_{\sL^{\infty}(\R)}\Big)\)
\[
2V'(x)+xV''(x)\leq c\in\R_{<0}\text{ or } 2V'(x)+xV''(x)\geq c\in\R_{>0}.
\]
Due to the assumption that \(V'\leqq 0\) the strict convexity is much more restrictive.

Thus, for strictly convex flux Oleinik's states that
\[
\exists C\in\R_{\geq0}: \partial_{x}q(t,x)\leq \tfrac{C}{t}\ \forall t\in[0,T],\ x\in\R \text{ a.e.}
\]
which is satisfied if \(q_{0}\) is OSL (one sided Lipschitz-continuous) from above, i.e.,\
\[
\partial_{x}q_{0}(x)\leq C\ \forall x\in\R
\] according to \cref{theo:monotonicity}. Thus, it particular holds for all monotonically decreasing datum.
For strictly concave flux, Oleinik's entropy condition states that
\[
\exists C\in\R_{\geq0}: \partial_{x}q(t,x)\geq \tfrac{-C}{t}\ \forall t\in[0,T],\ x\in\R \text{ a.e.}
\]
which is according to \cref{theo:monotonicity} always satisfied for monotonically decreasing initial datum, and also for monotonically increasing initial datum which is OSL from below.

However, as \cref{theo:convergence_monotone} is more general than Oleinik's entropy condition as we do not require additional assumptions on the velocity, we have only detailed \cref{theo:convergence_monotone}.
\end{remark}

\section{Numerical illustration}\label{sec:numerics}
In this section, we illustrate the convergence for different setups and compare in particular \textbf{nonlocal in the solution} with \textbf{nonlocal in the velocity} (see also \cref{eq:local_nonlocal}).

Let us first comment on \cref{fig:const_kernel} where we have chosen a constant kernel (note that our results are not applicable in this specific setting) with initial datum \(q_{0}\equiv\tfrac{1}{4}+\tfrac{1}{2}\chi_{[-0.5,0.5]}\). From top to bottom one can observe the claimed convergence. In the first and third column, we have the nonlocal operator acting on the solution and not on the velocity, while on the second and fourth column, on the velocity, one time for a concave velocity, and one time for an convex velocity. As can be seen, only for larger \(\eta\) (the first row) a difference between nonlocal in the solution and nonlocal in the velocity can be spotted, but for smaller \(\eta\) this difference vanishes for the eye.

We underline that our results and in particular \cref{theo:main} are not applicable to the problem (as they involve the constant kernel and not the exponential one), however, one can even in this case observe the convergence to the local entropy solution which cannot be explained with a viscosity effect of the underlying numerical scheme as we work with a characteristic based method capable of tracking the discontinuities precisely \cite{pflug2}.

The same numerical experiment is made for the exponential case (where \cref{theo:main} is applicable) in \cref{fig:1}. One can see the smoothing effect of the exponential kernel quite well for larger \(\eta\).

In \cref{fig:solution_time_slice} the solutions are plotted at time \(t=0.5\) for the different settings described above for the constant kernel. One can clearly see the violation of the monotonicity in for $x>0$ in the top left figure.

For any initial datum $q_0 \in \sL^\infty(\R)$ for which there exists $\bar x \in \R$ s.t.\ $q_0|_{(\bar x,\infty)}$ is monotone, the solution $q$ for the setting nonlocal in velocity is also monotone on the domain to the right of the characteristics emanating from $(0,\bar x) \in \OT$. This is a direct consequence of the method of characteristics (see e.g. \cite{pflug}) and \cref{rem:monotonicity}.
This is in general not true for the nonlocal in the solution formulation as already shown in \cite[Theorem 4.13, Theorem 4.18]{pflug4}. This can clearly be observed in all but the lower left plots in \cref{fig:solution_time_slice}.

Although we have not chosen a monotone initial datum one can roughly expect a monotonicity preserving numerical result for the nonlocal in the velocity case independent of the velocity. Indeed, as can be observed, the monotonicity is always preserved which is in contrast to the local in the solution case, where such a monotonicity is destroyed if the velocity is not satisfying specific assumptions.

\begin{figure}
    \centering
    \includegraphics[scale = 0.52,clip,trim=0 25 17 10]{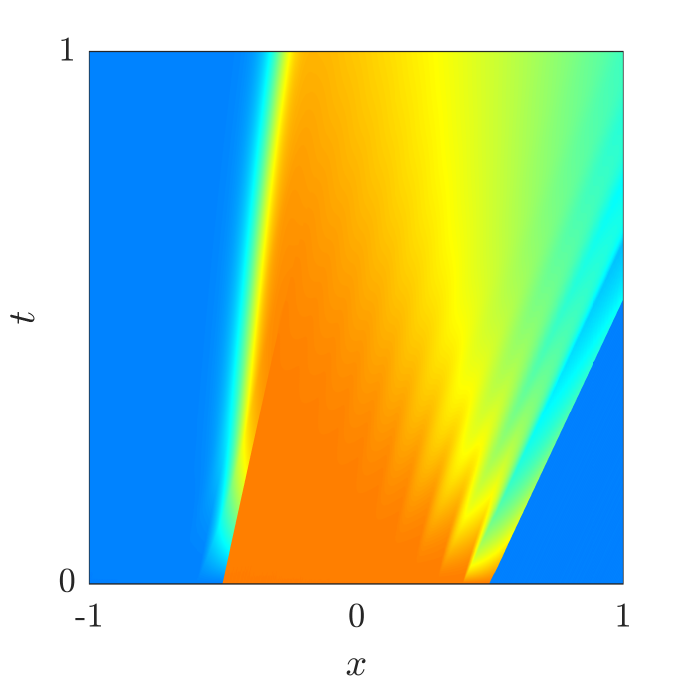}
   \includegraphics[scale = 0.52,clip,trim=22 25 17 10]{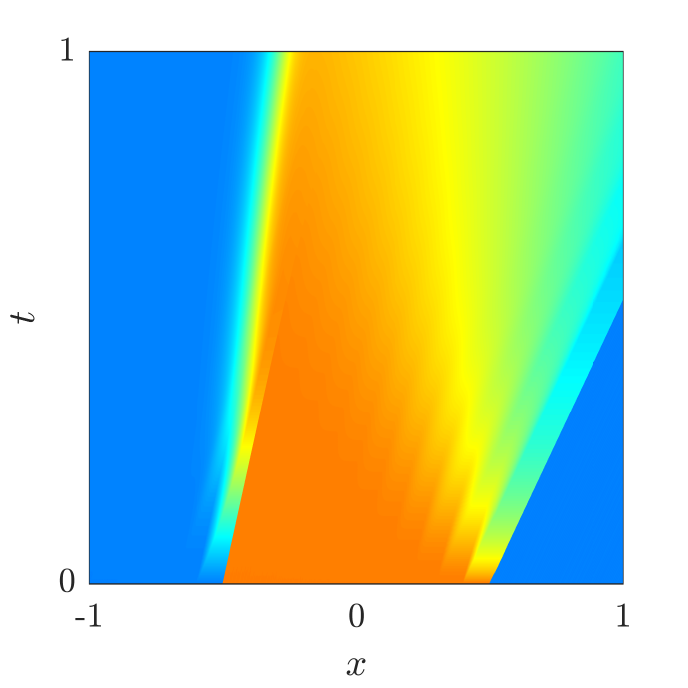}
    \includegraphics[scale = 0.52,clip,trim=22 25 17 10]{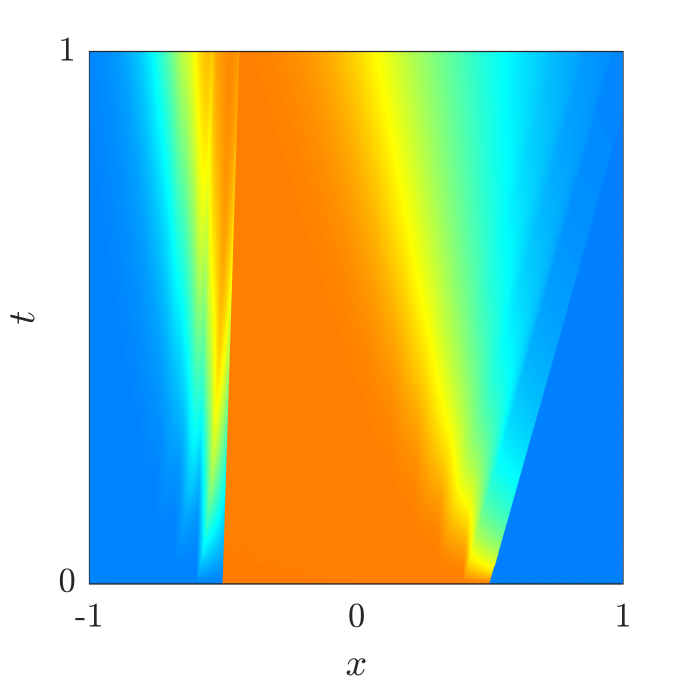}
    \includegraphics[scale = 0.52,clip,trim=22 25 17 10]{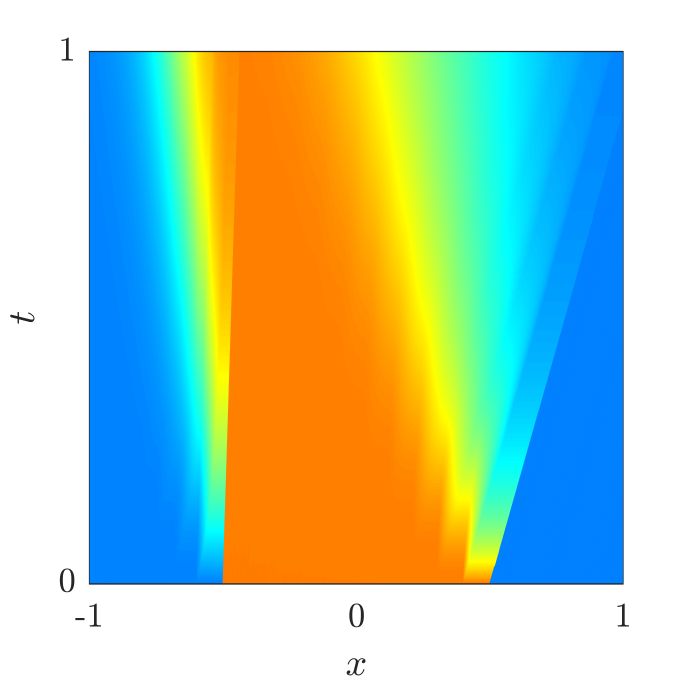}

    \includegraphics[scale = 0.52,clip,trim=0 25 17 10]{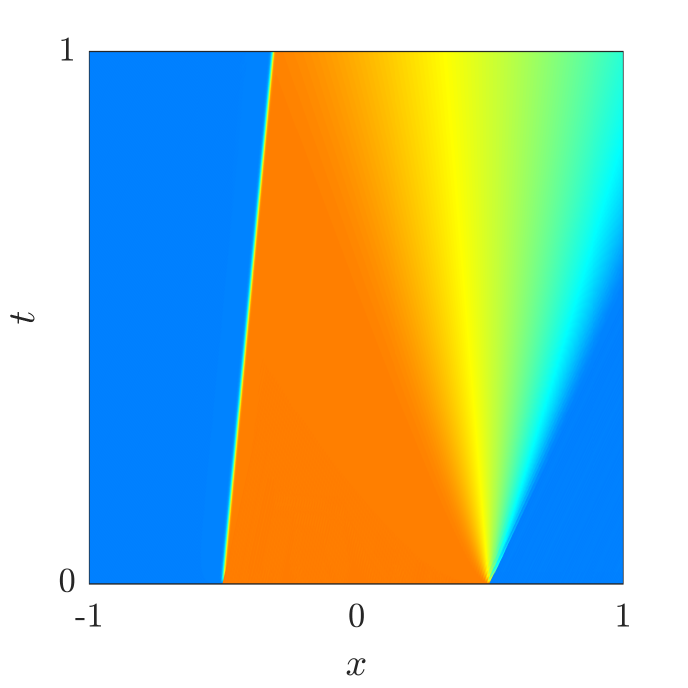}
   \includegraphics[scale = 0.52,clip,trim=22 25 17 10]{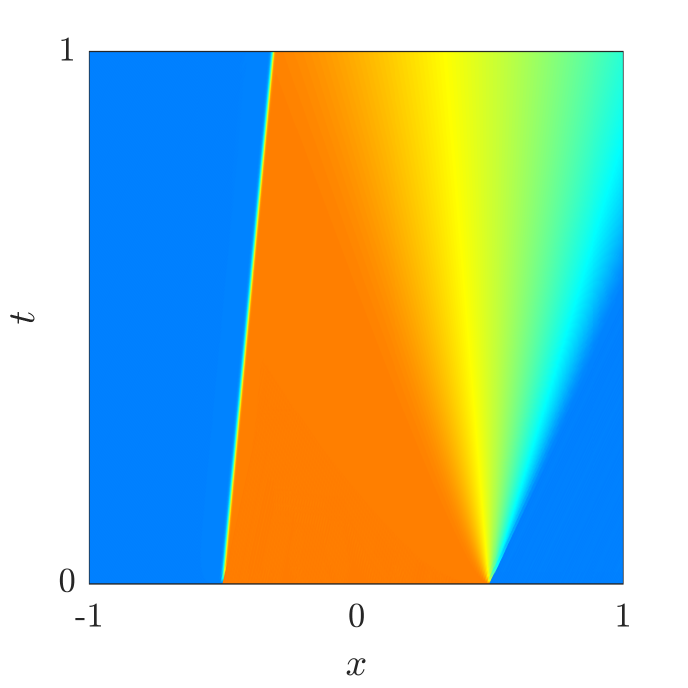}
    \includegraphics[scale = 0.52,clip,trim=22 25 17 10]{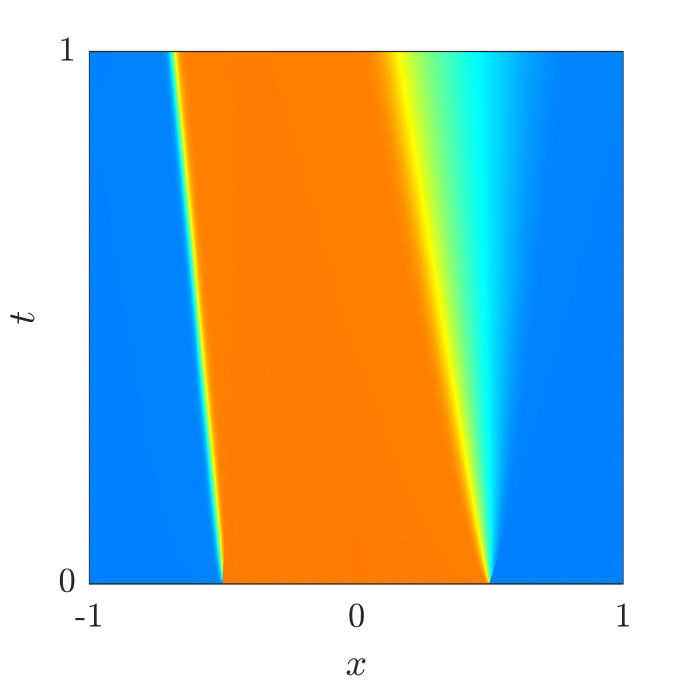}
    \includegraphics[scale = 0.52,clip,trim=22 25 17 10]{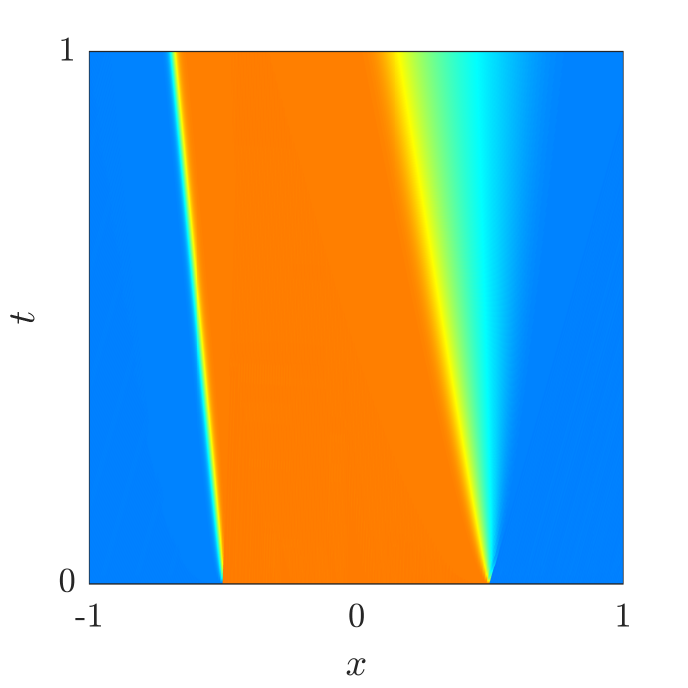}

    \includegraphics[scale = 0.52,clip,trim=0 0 17 10]{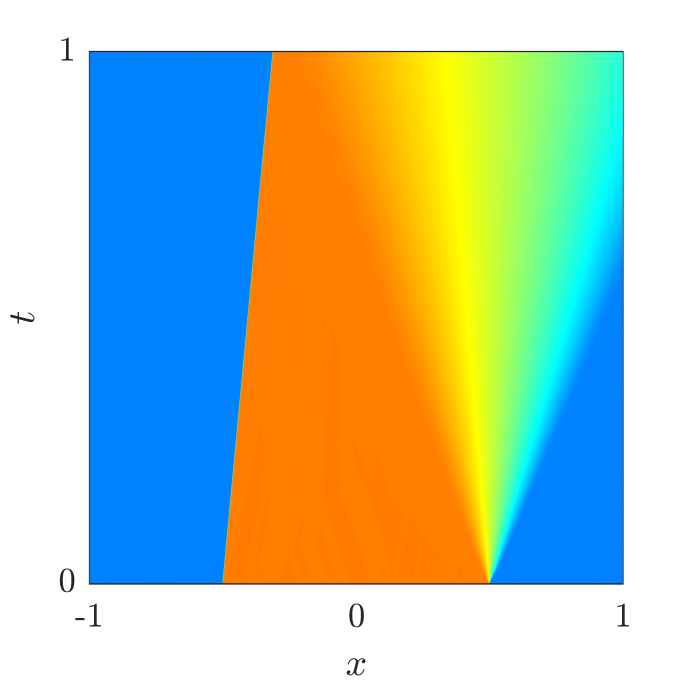}
   \includegraphics[scale = 0.52,clip,trim=22 0 17 10]{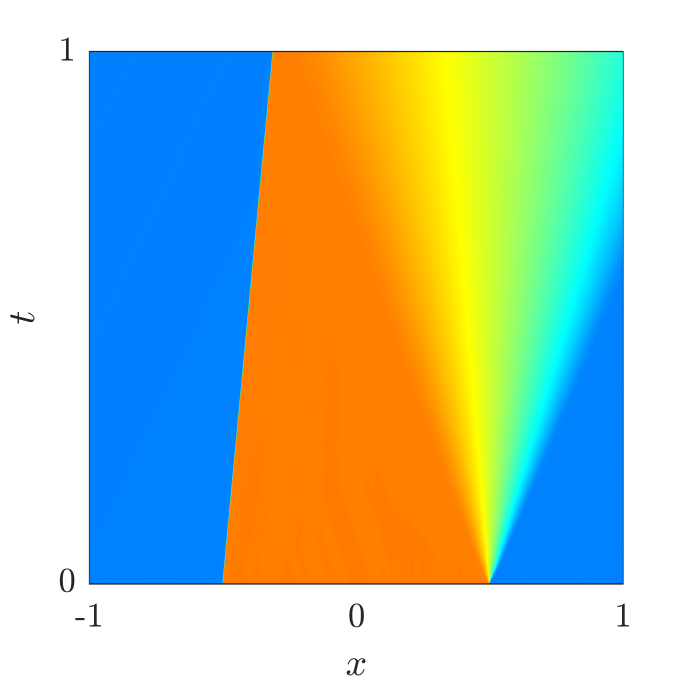}
    \includegraphics[scale = 0.52,clip,trim=22 0 17 10]{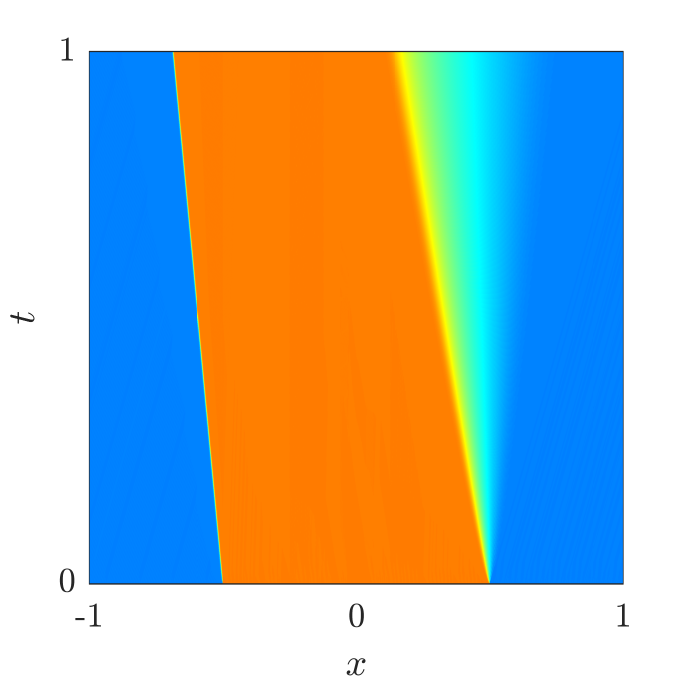}
    \includegraphics[scale = 0.52,clip,trim=22 0 17 10]{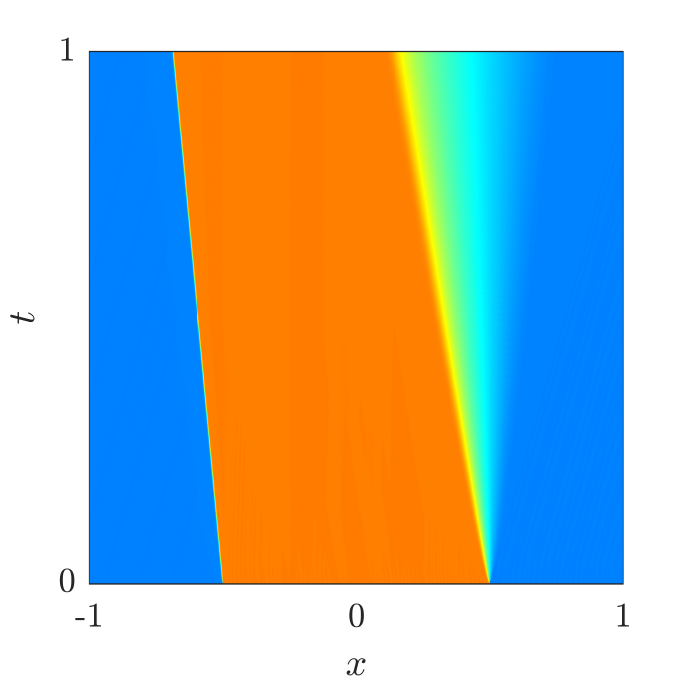}

    \caption{constant kernel $\gamma\equiv \eta^{-1}\chi_{(0,1)}(\cdot \eta^{-1})$ for $x\in \R$, initial datum \(q_{0}\equiv\tfrac{1}{4}+\frac{1}{2}\chi_{[-0.5,0.5]}\), \textbf{first column:} Nonlocal in the solution, \(V\equiv (1-\cdot)^2\), \textbf{second column:} Nonlocal in the velocity, \(V = (1-\cdot)^2,\) \textbf{third column:} nonlocal in the solution, \(V(\cdot)=1-(\cdot)^{2}\), \textbf{last column:} Nonlocal in the velocity, \(V(\cdot)=1-(\cdot)^{2}\). \textbf{Top row:} $\eta = 10^{-1}$, \textbf{middle row:} $\eta = 10^{-2}$ and \textbf{bottom row:} $\eta = 10^{-3}$, \textbf{Colorbar:}  $0\ $\protect\includegraphics[width=1.5cm]{ColorbarJet.png}$\, 1$}
    \label{fig:const_kernel}
\end{figure}

\begin{figure}
    \centering
    \includegraphics[scale = 0.52,clip,trim=15 40 20 10]{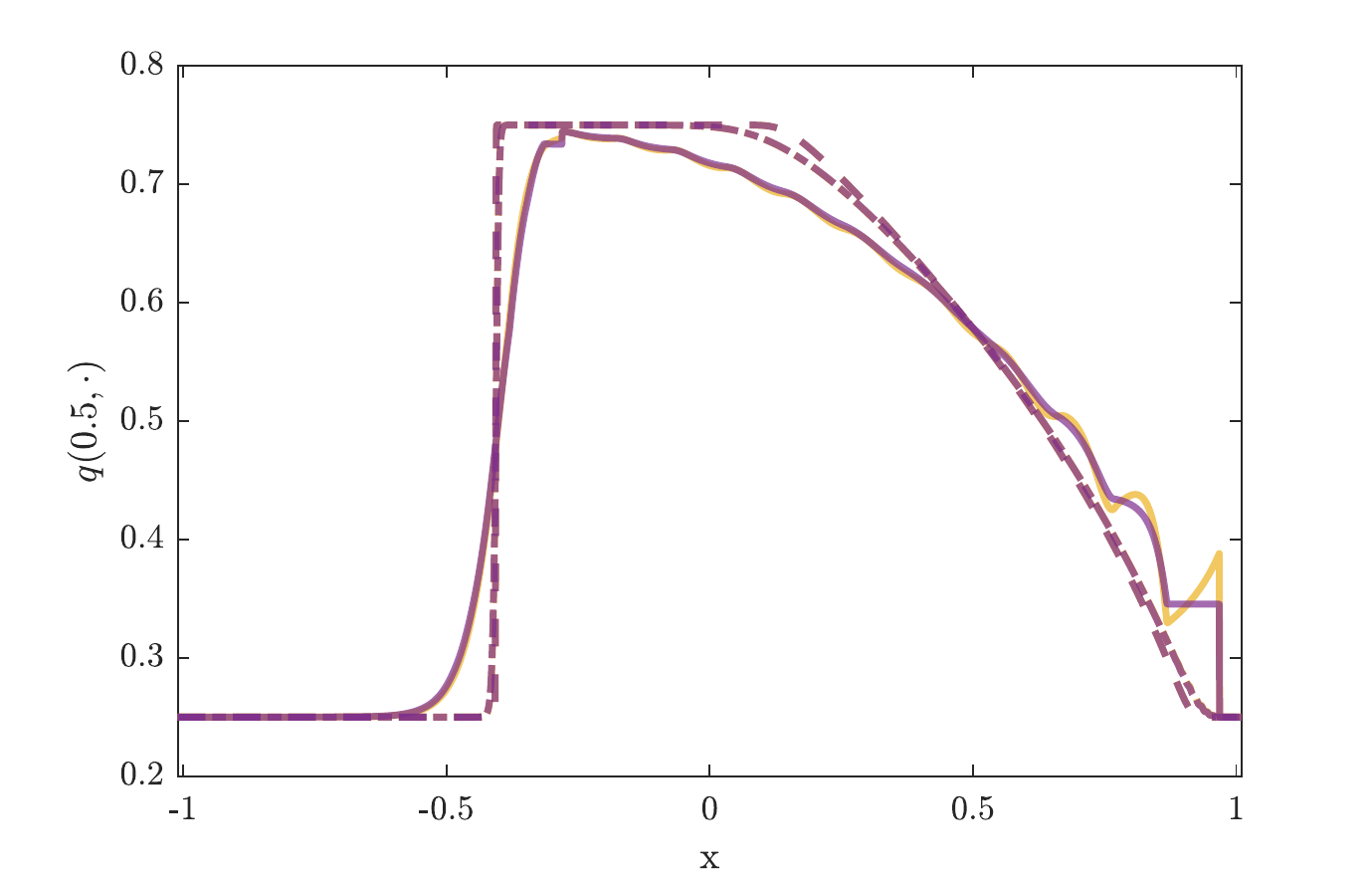}
    \includegraphics[scale = 0.52,clip,trim=50 40 20 10]{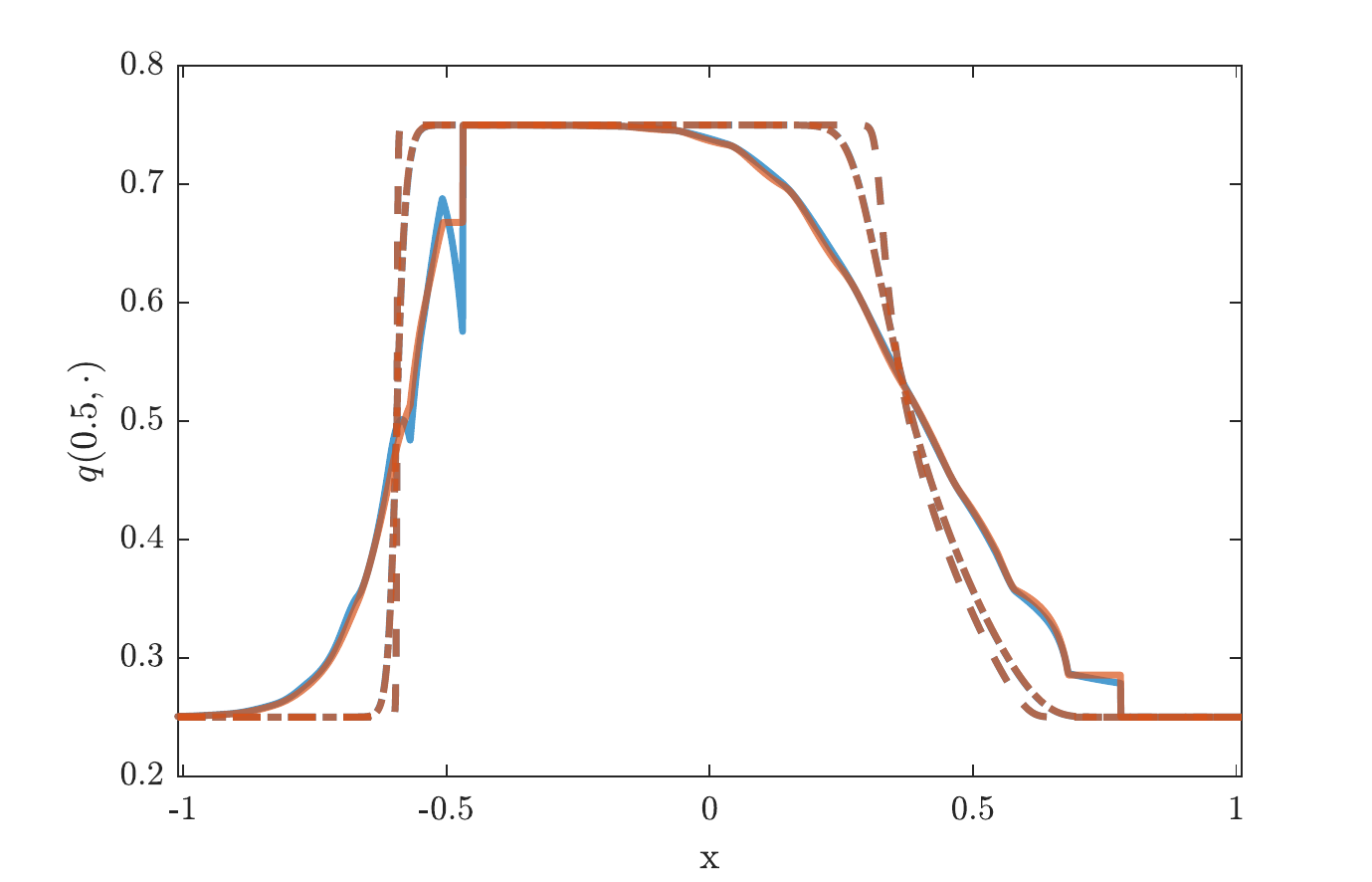}

    \includegraphics[scale = 0.52,clip,trim=15 0 20 10]{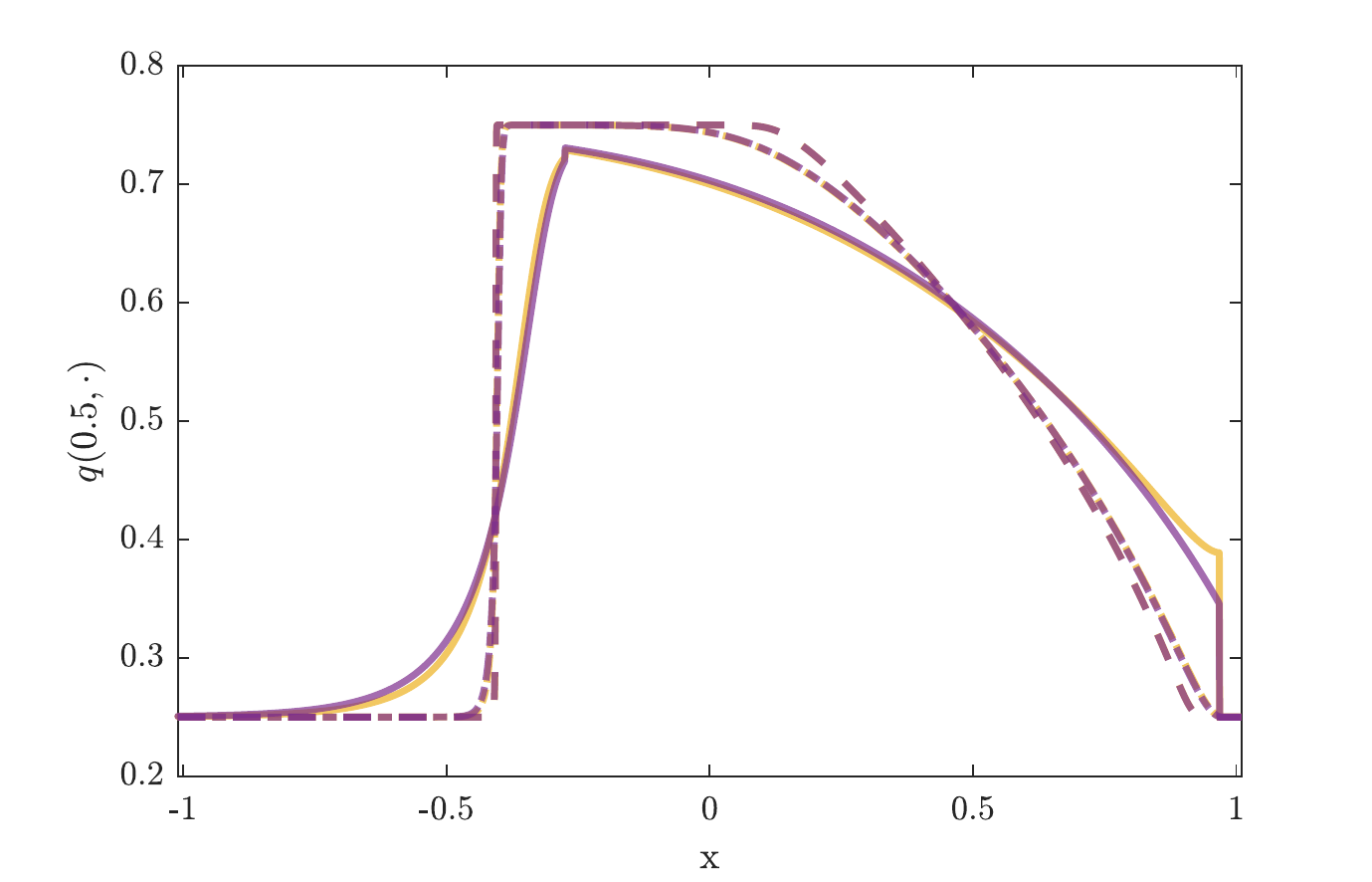}
    \includegraphics[scale = 0.52,clip,trim=50 0 22 10]{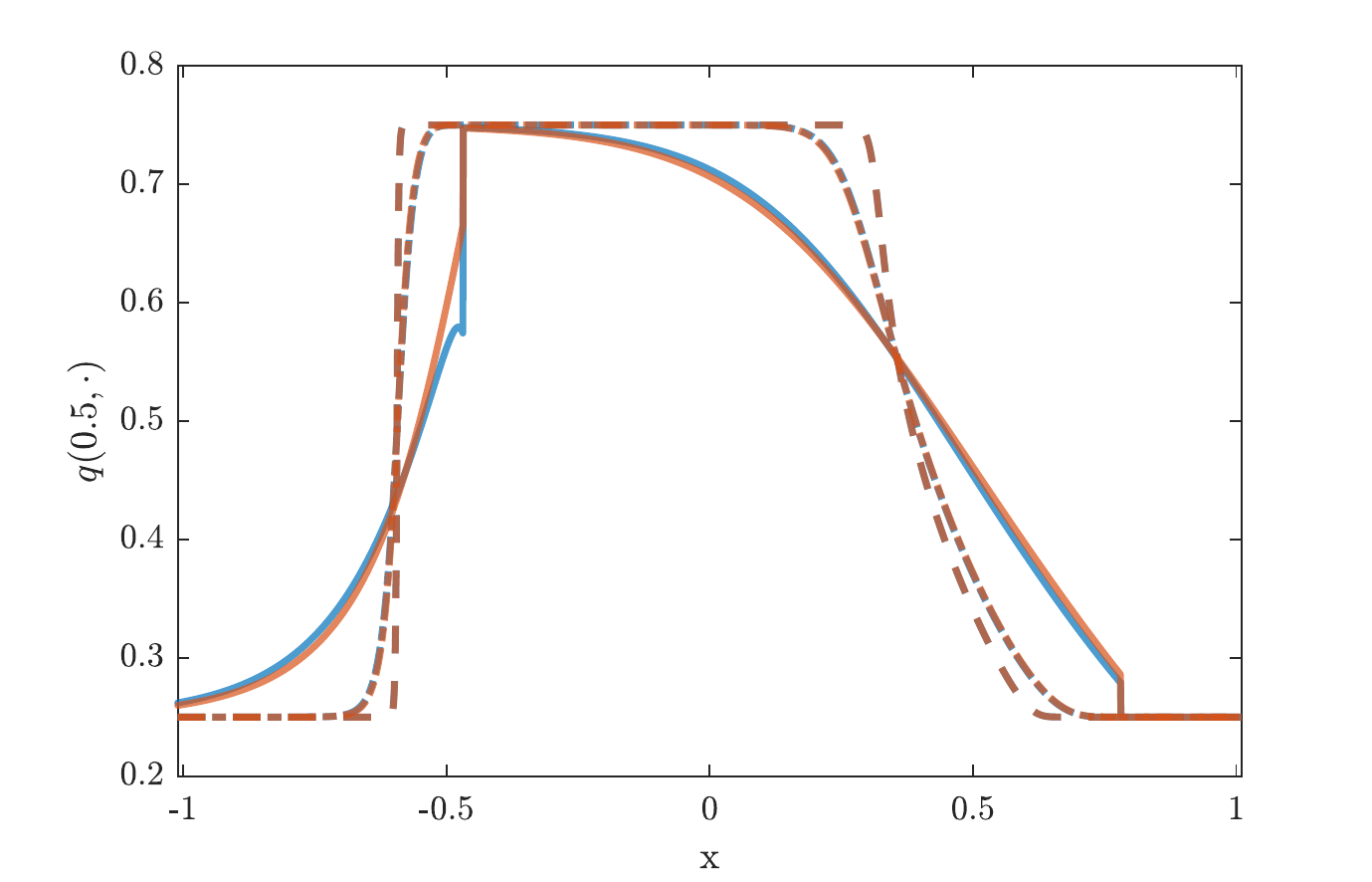}
    \caption{Same setup as in \crefrange{fig:1}{fig:const_kernel}. \textbf{Top row:} constant kernel, \textbf{bottom row:} exponential kernel, \textbf{left column:} convex velocity \(V\equiv (1-\cdot)^{2}\), \textbf{right column}: concave velocity \(V\equiv 1-(\cdot)^{2}\), \textbf{solid line:} $\eta=10^{-1}$, \textbf{dashed-dotted:} $\eta = 10^{-2}$ and \textbf{dashed:} $\eta = 10^{-3}$, \textbf{purple and red:} nonlocal in velocity, \textbf{yellow and blue}: nonlocal in solution.}
    \label{fig:solution_time_slice}
\end{figure}

\section{Conclusions and future work}\label{sec:conclusions}
In this contribution we have studied the convergence to the local entropy solutions when the nonlocal operator does not act on the solution itself but on the velocity. We have established for the exponential kernel and a variety of velocities functions the convergence to the local entropy solution. Moreover, for monotone datum and arbitrary kernels, we obtain this convergence as well. This motivates to study several open problems:
\textbf{1}) Does the convergence generally hold? It seems like the nonlocal averaging over the velocity behaves somewhat more reasonable as it conserves monotonicity, etc., however, on the other hand the results are not as general as in \cite{coclite2020general}. In our convergence analysis we use a specific strictly convex entropy depending on the velocity. This choice of entropy requires the velocity's derivative to be nonzero, and restricts the result. Is there another entropy which does not require this assumption and can we thus generalize the result?
\textbf{2}) In a recent preprint, the results for the singular limit problem nonlocal in the solution could be generalized to a variety of other (physical relevant) kernels \cite{keimer42,colombo2022nonlocal} and it would be important to understand whether this is also possible in the nonlocal in the velocity case.
\textbf{3)} Can we obtain the same results when considering instead of a conservation law a balance law with nonlinear right hand side?
\textbf{4)} In the case of a symmetric kernel, the solutions are not even uniformly bounded with respect to the nonlocal parameter. However, it seems that convergence in a weak sense yet to be determined might still hold and would give an even more general convergence result.
\textbf{5)} Control and optimal control in the singular limit case.

\section*{Acknowledgments}
Lukas Pflug thanks for the support by the Collaborative Research Centre 1411 “Design of Particulate Products” (Project-ID 416229255).
Jan Friedrich was supported by the German Research Foundation (DFG) under grant HE 5386/18-1, 19-2, 22-1, 23-1 and
Simone G\"ottlich under grant GO 1920/10-1.
\bibliographystyle{siamplain}
\bibliography{biblio}
\end{document}